\documentclass[11pt,reqno]{amsart} 

\usepackage{amssymb,enumitem,verbatim,mathtools,stmaryrd}
\usepackage[bookmarksdepth=3]{hyperref}
\usepackage[margin=1in]{geometry}
\usepackage[all,cmtip]{xy}
\usepackage[sort,noadjust]{cite}

\usepackage{dynkin-diagrams}

\usepackage[color=yellow]{todonotes}
\usepackage{multicol}

\newcommand{\subgrp}[1]{\langle #1 \rangle}

\newcommand{\set}[1]{\left\{ #1 \right\}}
\newcommand{\smallset}[1]{\{ #1 \}}

\newcommand{\abs}[1]{| #1 |}

\newcommand{\ol}[1]{\overline{#1}}
\newcommand{\ul}[1]{\underline{#1}}

\DeclareMathOperator{\End}{End}
\DeclareMathOperator{\Grp}{Grp}
\DeclareMathOperator{\Hom}{Hom}
\DeclareMathOperator{\id}{id}
\DeclareMathOperator{\Id}{id}
\DeclareMathOperator{\im}{im}
\DeclareMathOperator{\Ind}{Ind}

\DeclareMathOperator{\Irr}{Irr}
\DeclareMathOperator{\Lie}{Lie}
\DeclareMathOperator{\refl}{refl}

\DeclareMathOperator{\Res}{Res}
\DeclareMathOperator{\rk}{rk}

\DeclareMathOperator{\Span}{span}
\DeclareMathOperator{\tr}{tr}
\DeclareMathOperator{\triv}{triv}

\newcommand{\zero}{\ol{0}}
\newcommand{\one}{\ol{1}}

\newcommand{\ve}{\varepsilon}

\newcommand{\C}{\mathbb{C}}

\newcommand{\Q}{\mathbb{Q}}
\newcommand{\R}{\mathbb{R}}

\newcommand{\bbW}{\mathbb{W}}
\newcommand{\Z}{\mathbb{Z}}

\newcommand{\CB}{\C \calB}
\newcommand{\CD}{\C D}
\newcommand{\CG}{\C G}
\newcommand{\CH}{\C H}

\newcommand{\CS}{\C \fS}

\newcommand{\CSrm}{\C S}

\newcommand{\RG}{\R G}

\newcommand{\calB}{\mathcal{B}}
\newcommand{\calC}{\mathcal{C}}

\newcommand{\calE}{\mathcal{E}}
\newcommand{\calF}{\mathcal{F}}

\newcommand{\fb}{\mathfrak{b}}
\newcommand{\fD}{\mathfrak{D}}

\newcommand{\fg}{\mathfrak{g}}
\newcommand{\g}{\mathfrak{g}}
\newcommand{\gl}{\mathfrak{gl}}
\newcommand{\fh}{\mathfrak{h}}
\newcommand{\osp}{\mathfrak{osp}}

\newcommand{\fs}{\mathfrak{s}}
\newcommand{\fS}{\mathfrak{S}}
\newcommand{\fsl}{\mathfrak{sl}}
\newcommand{\fso}{\mathfrak{so}}
\newcommand{\fsp}{\mathfrak{sp}}
\newcommand{\fq}{\mathfrak{q}}
\newcommand{\sq}{\mathfrak{sq}}

\newcommand{\glmn}{\gl(m|n)}

\newcommand{\gone}{\g_{\one}}
\newcommand{\gzero}{\g_{\zero}}
\newcommand{\Gone}{G_{\one}}
\newcommand{\Gzero}{G_{\zero}}
\newcommand{\Hone}{H_{\one}}
\newcommand{\Hzero}{H_{\zero}}
\newcommand{\hzero}{\fh_{\zero}}

\newcommand{\Vzero}{V_{\zero}}

\newcommand{\CGzero}{\CG_{\zero}}
\newcommand{\CHzero}{\CH_{\zero}}

\numberwithin{equation}{subsection}
\counterwithin{figure}{section}
\counterwithin{table}{section}

\newtheorem{theorem}{Theorem}[subsection]
\newtheorem{proposition}[theorem]{Proposition}

\newtheorem{corollary}[theorem]{Corollary}
\newtheorem{lemma}[theorem]{Lemma}

\theoremstyle{definition}

\newtheorem{convention}[theorem]{Convention}

\newtheorem{remark}[theorem]{Remark}

\newtheorem{convention-construction}[theorem]{Convention/Construction}
\newtheorem{remark-convention}[theorem]{Remark/Convention}

\title[Lie (super)algebras generated by reflections]{Lie (super)algebras generated by reflections in finite reflection groups of exceptional type}

\author{Christopher M.\ Drupieski}
\address{Department of Mathematical Sciences,
		DePaul University,
		Chicago, IL 60614, USA}
\email{c.drupieski@depaul.edu}

\author{Jonathan R.\  Kujawa}
\address{Department of Mathematics \\
		Oregon State University \\
		Corvallis, OR 97331, USA}
\email{kujawaj@oregonstate.edu}

\thanks{JRK was supported by a Simons Foundation Travel Support for Mathematicians Grant (SFI-MPS-TSM-00014272).}

\date{\today}


\allowdisplaybreaks

\begin{document}

\begin{abstract}  Using a mixture of theoretical tools and computer calculations, we describe the Lie subalgebras and the Lie sub-superalgebras generated by the sets of reflections in the group algebras of the exceptional finite reflection groups.
\end{abstract}



\maketitle

\section{Introduction}

\subsection{Background}

Let $G$ be a finite reflection group, and let $\CG$ be its complex group algebra. The commutator bracket $[x,y] = xy - yx$ endows $\CG$ with the structure of a complex Lie algebra, which we denote $\Lie(\CG)$. The classical theorems of Maschke and of Artin and Wedderburn immediately imply that $\Lie(\CG)$ is a direct sum of general linear Lie algebras, one for each isomorphism class of simple $\CG$-modules. A more subtle and interesting problem is to determine the structure of the Lie subalgebra $\fs \subseteq \Lie(\CG)$ generated by the set $S$ of reflections in $G$. In \cite{Marin:2007}, Marin considered the case $G = \fS_n$ of the symmetric group, calculating the structure of the Lie subalgebra $\fs \subseteq \Lie(\CS_n)$ generated by the set of transpositions in $\fS_n$. In that paper and its sequels, Marin and others used the structure of $\fs$ to obtain new results on the representation theory of braid groups and Iwahori--Hecke algebras in type A.

By declaring the reflections in $G$ to be of odd super\-degree, one can also endow $G$ with the structure of a super\-group, i.e., a group equipped with a multiplicative grading by $\Z_{2} = \Z /2\Z$. Denote the $\Z_2$-degree of a homogeneous element $x$ by $\ol{x}$. The complex group algebra $\CG$ is then endowed in turn with an associative super\-algebra structure, and the graded commutator
	\[
	[x,y] = xy - (-1)^{\ol{x} \cdot \ol{y}}yx
	\]
defines a Lie super\-algebra structure on $\CG$, which we again denote $\Lie(\CG)$. One can then ask for an explicit description of the Lie superalgebra $\Lie(\CG)$, and for a description of the Lie sub-super\-algebra $\g$ generated by the set of reflections in $G$.

In \cite{DK:2025}, we determined for $G = \fS_n$ the structure of the Lie superalgebra $\Lie(\CG)$, and the structure of the Lie sub-superalgebra generated by the set of reflections in $G$. Then in subsequent work \cite{DK:2026a,DK:2026b}, we determined the structure of the corresponding Lie (super)algebras for the finite reflection groups of types BC and D. While the overall approach we developed for these remaining classical types was uniform, and simplified and streamlined (and subsumed) the work from type A, some modest case-by-case verifications were still required.

\subsection{Main results}

Given these results, it is natural to ask what happens for the exceptional finite reflection groups, i.e., for the irreducible reflection groups of Coxeter types $I_{2}(m)$ ($m\geq 3$), $H_{3}$, $H_{4}$, $F_{4}$, $E_{6}$, $E_{7}$, and $E_{8}$. In each case, describing $\Lie (\C G)$ remains straightforward, and the main (hard) problem is to describe the Lie subalgebra $\fs$ and Lie sub-superalgebra $\fg$ generated by the set $S$ of reflections in $G$. The answers vary in the details, but the overall shape of the results is similar across cases. We outline this shape below.

First, let $\CSrm$ be the vector subspace of $\CG$ spanned by the set $S$, and let $(\CSrm)^G$ be subspace of elements that are invariant under conjugation by $G$. Then $(\CSrm)^G$ admits a basis given by the conjugacy class sums of elements of $S$. In particular, if $G$ is an irreducible finite reflection group, then $(\CSrm)^G$ is at most two-dimensional.

For the Lie subalgebra $\fs \subseteq \Lie(\CG)$ generated by the set of reflections in $G$, one gets that $\fs$ is a reductive complex Lie algebra, its derived subalgebra $\fD(\fs)$ is a semisimple complex Lie algebra, and its center $Z(\fs)$ is equal to $(\CSrm)^G$. Thus
	\[
	\fs = \fD (\fs) \oplus (\CSrm)^G,
	\]
and it suffices to describe $\fD (\fs)$. Let $V_{\refl}$ denotes the reflection representation of $G$ and let $\varepsilon$ denote the sign representation of $G$, i.e., the one-dimensional representation that satisfies $\varepsilon(s)=-1$ for all $s \in S$. Our first collection of main results gives explicit sets of simple $\CG$-modules, $\mathtt{E}=\mathtt{E}(G)$ and $\mathtt{F}=\mathtt{F}(G)$, such that\footnote{In type $H_{4}$, an additional summand $\fsl({}^{\sigma}V_{\refl})$ must also be included.}
	\begin{equation} \label{eq:D(s)-general-iso}
	\fD(\fs) \cong \fsl(V_{\refl}) \oplus \Big[ \bigoplus_{V \in \mathtt{E}/\sim} \fsl(V) \Big] \oplus \Big[ \bigoplus_{V \in \mathtt{F}} \osp(V) \Big].
	\end{equation}
Here $\osp(V)$ is shorthand for a direct summand that is either an orthogonal or a symplectic Lie algebra. The isomorphism \eqref{eq:D(s)-general-iso} arises from a restriction of the Artin--Wedderburn isomorphism for the semisimple group algebra $\CG$. The sets $\mathtt{E}$ and $\mathtt{F}$ vary slightly depending on the Coxeter type of $G$ (and may be empty), but they always satisfy the properties
	\begin{itemize}
	\item if $V \in \mathtt{F}$, then $V \otimes \varepsilon \cong V$ as $\CG$-modules, and
	\item if $V \in \mathtt{E}$, then $V \otimes \varepsilon \not\cong V$ as $\CG$-modules.
	\end{itemize}
The summation over elements of $V \in \mathtt{E}/\!\!\sim$ means that we take a direct sum over a (any) complete set of representatives in $\mathtt{E}$ for the equivalence relation defined by $V \sim V \otimes \ve$.

Now consider the super case. Using the graded version of the Artin--Wedderburn theorem, $\Lie (\CG)$ can be identified with a direct sum of isomeric (type Q) and general linear (type M) Lie super\-algebras corresponding to the simple $\CG$-super\-modules of types Q and M, respectively:
	\begin{equation} \label{eq:Lie(CG)-super-general}
	\Lie (\CG) \cong 
	\Big[ \bigoplus_{W \in \Irr_Q(G)} Q(W) \Big]
	\oplus 
	\Big[ \bigoplus_{W \in \Irr_M(G)} \End(W) \Big]. 
	\end{equation}
Let $\fD(\Lie(\CG))$ denote the derived Lie sub-superalgebra of $\Lie(\CG)$, and let $\g \subseteq \Lie(\CG)$ be the Lie sub-superalgebra generated by the set of reflections in $G$. Our second collection of main results demonstrates that $\fg$ has a remarkably uniform description:
	\begin{equation} \label{eq:g-super-general}
	\g = \fD(\Lie(\CG)) + (\CSrm)^G.
	\end{equation}

It is worth highlighting the difference between the super and non-super settings. For example, the dimension of $\g$ is roughly twice the dimension of $\fs$. Combined with our previous work in \cite{DK:2025,DK:2026b}, the results of this paper also show that \eqref{eq:g-super-general} holds for all \emph{irreducible} finite reflection groups. We can offer no a priori explanation for why \eqref{eq:g-super-general} should hold across all irreducible types, as we are only able to deduce the equality as a consequence of inductive or computer calculations whose details vary by Coxeter type. Further, dimension calculations indicate that \eqref{eq:g-super-general} does not necessarily hold if $G$ is a reducible finite reflection group (e.g., when $G$ contains multiple factors of Coxeter type B), and we have no general theory for the structure of $\g$ when $G$ is reducible. This contrasts with the non-super setting, where the structure of $\fs$ can be easily determined from the Coxeter types of the irreducible components of $G$; see Lemma \ref{lemma:s(W)-Lie-decomposition}.

\subsection{Outline of the paper}

We begin in Section~\ref{S:LA-by-reflections} by establishing a number of general results on Lie algebras in this context, including several that constrain the image of $\fs$ under the Artin--Wedderburn isomorphism. For example, Lemma~\ref{lemma:osp-factorization} shows that if $V \cong V \otimes \varepsilon$, then the image of $\fs$ in the Artin--Wedderburn summand corresponding to $V$ lies in a factor of the form $\osp(V)$. On the other hand, Section~\ref{subsubsec:dual-pairs} shows that if $V \not\cong V \otimes \varepsilon$, then one can omit one or the other of $V$ or $V \otimes \varepsilon$ while still maintaining injectivity of the Artin--Wedderburn map on $\fD(\fs)$.  Likewise, Corollary~\ref{cor:exterior-powers-factorization} shows that one can omit the simple modules $\Lambda^{d}(V_{\refl})$ for $d > 0$ while still maintaining injectivity of the Artin--Wedderburn map on $\fD(\fs)$. It follows that the Artin--Wedderburn map restricts to an injective Lie algebra morphism
	\[
	\fD(\fs) \hookrightarrow \fsl(V_{\refl}) \oplus \Big[ \bigoplus_{V \in E/\sim} \fsl(V) \Big] \oplus \Big[ \bigoplus_{V \in F} \osp(V) \Big]
	\]
for sets $E$ and $F$ specified in Section~\ref{SS:summary-and-injection}. This morphism may or may not be an isomorphism. From this stage forward, each Coxeter type must be handled via a case-by-case analysis involving a mixture of representation theory and computer calculations. The sets $E$ and $F$ are further reduced for each type until an isomorphism is achieved; see Section~\ref{S:Lie-algebra-calcs} for details. It should be noted that in types $E_{6}$ and $E_{8}$ there is a lingering ambiguity for certain simple modules $V$ on whether $\osp(V)$ is an orthogonal Lie algebra or a symplectic Lie algebra. This can be determined by a study of a bilinear form or by direct computer calculation. However, since doing so would take some effort and we did not need that information, we leave it to the interested reader.

Next, Section~\ref{sec:preliminaries} contains a number of general results on Lie superalgebras in this context. This includes a classification of the simple $\C G$-supermodules, and, via the graded analogue of the Artin--Wedderburn theorem, the identification \eqref{eq:Lie(CG)-super-general} of $\Lie(\CG)$. In Section~\ref{subsec:Lie-superalgebra-generalities} we identify a list of properties that play a key role in studying the Lie sub-superalgebra $\fg$.  For example, these properties ensure that
	\[
	\g \subseteq  \fD(\Lie (\CG)) + (\CSrm)^G.
	\]
They also ensure that $\fg_{\ol{0}}$ (the even subspace of $\g$) is a reductive Lie algebra. This Lie algebra, its subalgebras, and its representation theory play a key role in showing that the above inclusion is in fact an equality. As in the non-super setting, the arguments involve type-by-type analysis involving a mixture of representation theory and computer calculations, which we carry out in Section \ref{S:Lie-superalgebra-calcs}. The careful reader will note that some of the properties listed in Section \ref{subsec:Lie-superalgebra-generalities} fail to hold for arbitrary finite reflection groups. These properties thus appear to be somewhat essential for the validity of the equality \eqref{eq:g-super-general}.

The computer calculations for this paper were conducted using GAP \cite{GAP4}. The code we used to carry out these calculations is included as an ancillary file with the copy of this paper posted on the arXiv.

\subsection{Further directions}

At this stage, the natural next step is to follow the methods of  Marin and others for the symmetric group and to use $\fs$ and $\fg$ to study the Artin groups and Iwahori-Hecke algebras in the various types beyond type A. It would also be interesting to describe $\fs$ and $\fg$ when $G$ is an infinite reflection group, and to consider the case when the ground field has positive characteristic.  These settings would require new ideas. Our approach leans heavily on $\C G$ being a finite-dimensional semisimple (super)algebra.  On the other hand, we expect our methods to generalize in some way to finite complex reflection groups and to Lie color algebras (i.e., graded Lie algebras where $\Z_{2}$ is replaced with $\Z/r\Z$ for some $r > 2$).

\subsection{Generalities} \label{subsec:generalities}

Let $G$ be a finite reflection group with set of reflections $S$. Equivalently, there exists a subset $S_0 \subseteq S$ of simple reflections such that $(G,S_0)$ is a finite Coxeter system. If $S_0 = \{s_1,\ldots,s_n \}$, then $G$ is generated by the set $S_0$ subject only to relations of the form $(s_is_j)^{m_{i,j}} = 1$, with $m_{i,i} = 1$ and $m_{i,j} = m_{j,i} \geq 2$ for $i \neq j$.

Let $\Irr(G) = \{ S^{\lambda}: \lambda \in \Lambda \}$ be a complete set of simple $\CG$-modules up to isomorphism; when it is necessary to emphasize the dependence of the indexing set $\Lambda$ on the group $G$, we may write $\Lambda = \Lambda(G)$. Let $\ve: G \to \set{\pm 1}$ be the one-dimensional sign representation of $G$ that is defined on reflections $s \in S$ by $\ve(s) = -1$. Then for each $\lambda \in \Lambda$ one has $S^{\lambda} \otimes \ve \cong S^{\lambda'}$ for a unique $\lambda' \in \Lambda$. Given a module $V$, we may also write $V^\ve$ for $V \otimes \ve$, so $(S^\lambda)^\ve \cong S^{\lambda'}$. Schur's lemma implies for each $\lambda \in \Lambda$ that, up to mutual rescaling, there exist unique linear maps $\phi^\lambda: S^\lambda \to S^{\lambda'}$ and $\phi^{\lambda'}: S^{\lambda'} \to S^\lambda$ (corresponding to $\CG$-module isomorphisms $S^\lambda \otimes \ve \cong S^{\lambda'}$ and $S^{\lambda'} \otimes \ve \cong S^{\lambda}$) such that for all $g \in G$, $v \in S^\lambda$, and $w \in S^{\lambda'}$,
	\begin{align}
	\phi^\lambda(g \cdot v) &= \ve(g) g \cdot \phi^\lambda(v), &
	\phi^{\lambda'}(g \cdot w) &= \ve(g) g \cdot \phi^{\lambda'}(w), \label{eq:associator} \\
	\phi^{\lambda'} \circ \phi^\lambda &= \id_{S^\lambda}, & 
	\phi^{\lambda} \circ \phi^{\lambda'} &= \id_{S^{\lambda'}}. \label{eq:inverse-associator}
	\end{align}
In the case $\lambda = \lambda'$, this specifies the map $\phi^\lambda : S^\lambda \to S^\lambda$ up to scalar multiplication by $\pm 1$. In the literature, a map $\phi^\lambda$ satisfying \eqref{eq:associator} is called an \emph{associator} for the isomorphism $S^\lambda \otimes \ve \cong S^{\lambda'}$.

Let $\CSrm \subset \CG$ be the subspace of $\CG$ spanned by $S$, and let
	\[
	(\CSrm)^G = \{ x \in \CSrm : gxg^{-1} = x \}
	\]
be the subspace of elements invariant under the conjugation action of $G$. Since $S$ is a union of conjugacy classes in $G$, say $S = C_1 \cup \cdots \cup C_t$, then $(\CSrm)^G$ admits the basis of class sums $\calC_1,\ldots,\calC_t$, where $\calC_i = \sum_{s \in C_i} s$. We will frequently assume that the Coxeter system $(G,S_0)$ is irreducible.\footnote{Recall that an arbitrary finite reflection group is a direct product of irreducible reflection groups; see \cite[1.3.1]{Geck:2000}.} In that case the set $S$ consists of at most two conjugacy classes, which we may denote as $C_\ell$ and $C_s$, referring to the elements of $C_\ell$ and $C_s$ as the \emph{long} and \emph{short} reflections, respectively. If $G$ has only one conjugacy class of reflections, then we take $C_s = \emptyset$ and we refer to all reflections in $G$ as long. Let $\calC_\ell = \sum_{s \in C_\ell} s$ and $\calC_s = \sum_{s \in C_s} s$, where by definition $\calC_s = 0$ if $C_s = \emptyset$.

Each ordinary irreducible representation of $G$ can be realized over a subfield of the real numbers; cf.\ \cite[5.3.8, 5.4.5, 5.5.6, 5.6.4]{Geck:2000}.\footnote{Suppose $G$ is irreducible. If $G$ is of type $H_3$ or $H_4$, then $G$ splits over $\Q(\sqrt{5})$. Otherwise, $G$ splits over $\Q$.} Then for each $\lambda \in \Lambda$, one has $S^\lambda = S_\R^\lambda \otimes_{\R} \C$ for a simple $\RG$-module $S_\R^\lambda$, and \cite[Lemma 7.4]{Lam:2001} implies that the map $\phi^\lambda: S^\lambda \to S^{\lambda'}$ can be realized via scalar extension from a linear isomorphism $\phi_\R^\lambda: S_\R^\lambda \to S_\R^{\lambda'}$ satisfying the analogue over $\R$ of \eqref{eq:associator}. It also follows that the maps $\phi_\R^\lambda$ and $\phi_\R^{\lambda'}$ can be chosen (by rescaling if necessary) such that
	\begin{equation}
	\phi_\R^{\lambda'} \circ \phi_\R^\lambda = \alpha_{S_\R^\lambda} \cdot \id_{S_\R^\lambda} \qquad \text{and} \qquad 
	\phi_\R^{\lambda} \circ \phi_\R^{\lambda'} = \alpha_{S_\R^{\lambda'}} \cdot \id_{S_\R^{\lambda'}},
	\end{equation}
with $\alpha_{S_\R^\lambda} = \alpha_{S_\R^{\lambda'}} \in \set{\pm 1}$. In the case $\lambda = \lambda'$, this specifies $\phi_\R^\lambda$ up to multiplication by $\pm 1$.


\section{Lie algebras generated by reflections}\label{S:LA-by-reflections}

In this section let $G$ be a finite reflection group with set of reflections $S$ as in Section \ref{subsec:generalities}.

\subsection{General structure of the Lie algebra generated by reflections} \label{subsec:Lie-algebra-generalities}

Write $\Lie(\CG)$ for the group algebra $\CG$ considered as a Lie algebra via the commutator bracket $[x,y] = xy-yx$. Let $\fs = \fs(G)$ be the Lie subalgebra of $\Lie(\CG)$ generated by the set $S$ of reflections in $G$.

\begin{remark} \label{remark:s-generated-by-S0}
If $x \in \Lie(\CG)$ and if $s \in G$ is an element of order $2$, then
	\[ \textstyle
	sxs = x - \frac{1}{2} \cdot [s,[s,x]] \in \Lie(\CG).
	\]
This implies that $\fs$ is also generated by the subset $S_0$ of simple reflections; cf.\ \cite[Remark 4.1.7]{DK:2026a}.
\end{remark}

The following lemma explains why it suffices to consider the case when $G$ is irreducible.

\begin{lemma} \label{lemma:s(W)-Lie-decomposition}
Let $G$ be a finite reflection group with set of reflections $S$. Let $G = G_1 \times \cdots \times G_r$ be a decomposition of $G$ into a product of irreducible finite reflection groups. Then as Lie algebras,
	\[
	\fs(G) = \fs(G_1) \oplus \cdots \oplus \fs(G_r).
	\]
\end{lemma}

\begin{proof}
Let $S_i = S \cap G_i$. Then $S$ is the disjoint union of the sets $S_1,\ldots,S_r$. For $i \neq j$, the elements of $G_i$ commute with the elements of $G_j$ in $\Lie(\CG)$. Then $\fs(G) = \fs(G_1) \oplus \cdots \oplus \fs(G_r)$.
\end{proof}


Let $\fD(\fs) = [\fs,\fs]$ be the derived subalgebra of $\fs$, and let $Z(\fs)$ be the center of $\fs$. Below we recall some general results on the structure of $\fs$; proofs can be found in \cite[\S4.1]{DK:2026a}.

\begin{lemma} \label{lemma:g-reductive}
With notation and assumptions as above, the Lie algebra $\fs$ is reductive, $\fD(\fs)$ is semisimple, $\fs = \fD(\fs) \oplus Z(\fs)$, and $Z(\fs) = (\CSrm)^G = \Span\set{\calC_1,\ldots,\calC_t}$.
\end{lemma}

Given a $\CG$-module $V$, let $\Res^{\CG}_{\fs}(V)$ be the $\fs$-module obtained by restriction.

\begin{lemma} \label{lemma:simples-remain-simple}
If $V \in \Irr(G)$, then $\Res^{\CG}_{\fs}(V)$ is simple as a module over $\fs$ and over $\fD(\fs)$.
\end{lemma}

If $V$ is a $\CG$-module, then the dual space $V^* = \Hom(V,\C)$ is a $\CG$-module with action defined for $g \in G$, $\phi \in V^*$, and $v \in V$ by $(g.\phi)(v) = \phi(g^{-1}.v)$. On the other hand, $[\Res^{\CG}_{\fs}(V)]^*$ is an $\fs$-module via the contragredient action, defined for $x \in \fs$, $\phi \in V^*$, and $v \in V$ by $(x.\phi)(v) = -\phi(x.v)$. Denote $V^*$ with this latter $\fs$-module structure by $V^{*,\Lie}$. Recall the notation $V^\ve = V \otimes \ve$.

\begin{lemma} \label{lemma:dual-tensor-sign-restriction}
Let $\ve: G \to \{\pm 1 \}$ be the sign representation of $G$. Then $\Res^{\CG}_{\fs}((V^*)^\ve) \cong V^{*,\Lie}$.
\end{lemma}

\begin{corollary} \label{cor:tensor-sign-restriction}
If $V$ is a finite-dimensional $\CG$-module, then $\Res^{\CG}_{\fs}(V^\ve) \cong V^{*,\Lie}$.
\end{corollary}

\begin{proof}
Each ordinary irreducible representation of $G$ can be realized over a subfield of $\R$, and hence each irreducible character of $G$ is real-valued. This implies for each finite-dimensional $\CG$-module $V$ that $V$ and $V^*$ have the same character, and hence $V \cong V^*$ as $\CG$-modules. The result then follows from the previous lemma and the functoriality of restriction from $\CG$ to $\fs$.
\end{proof}

\begin{proposition} \label{prop:iso-equivalent}
Let $V_1$ and $V_2$ be simple $\CG$-modules of dimension greater than one, and suppose the set $S$ consists of a single conjugacy class. Then the following are equivalent:
	\begin{enumerate}
	\item \label{item:G-iso} $V_1$ and $V_2$ are isomorphic as $\CG$-modules.
	\item \label{item:g-iso} $V_1$ and $V_2$ are isomorphic as $\fs$-modules.
	\item \label{item:D(g)-iso} $V_1$ and $V_2$ are isomorphic as $\fD(\fs)$-modules.
	\end{enumerate}
\end{proposition}

\begin{proof}
Argue as in the proofs of \cite[Proposition 4.6.3]{DK:2025}, \cite[Proposition 2.3.7]{DK:2026b}, or \cite[Proposition 2]{Marin:2007}; see also the proof of \cite[Proposition 4.2.5]{DK:2026a}.
\end{proof}

\subsection{Factorizations} \label{subsec:factorizations}

Given $V \in \Irr(G)$, let $\rho_V: \CG \to \End(V)$ be the corresponding module structure map. Then the Artin--Wedderburn isomorphism for $\CG$,
	\begin{equation} \label{eq:CG-Artin-Wedderburn}
	\CG \cong \bigoplus_{V \in \Irr(G)} \End(V), \qquad \sigma \mapsto \bigoplus_{V \in \Irr(G)} \rho_V(\sigma).
	\end{equation}
restricts to an injective Lie algebra homomorphism
	\begin{equation} \label{eq:D(s)-Artin-Wedderburn}
	\fD(\fs) \hookrightarrow \bigoplus_{V \in \Irr(G)} \fsl(V).
	\end{equation}
For $V \in \Irr(G)$, set $\fs_V = \rho_V(\fD(\fs)) \subseteq \fsl(V)$. If $\dim(V) = 1$, then $\fs_V = 0$.

\subsubsection{Dual pairs} \label{subsubsec:dual-pairs}

Let $V \in \Irr(G)$. By Corollary \ref{cor:tensor-sign-restriction}, $V^\ve \cong V^{*,\Lie}$ as $\fs$-modules. Then with respect to an appropriate choice of basis, the $\fs$-module structure map $\rho_{V^\ve}: \fs \to \gl(V^\ve)$ can be written in the form $X \mapsto -\rho_V(X)^t$, where $A^t$ denotes the transpose of $A$. This map evidently factors through $\rho_V$, and one gets the following commutative diagram:
	\begin{equation} \label{eq:dual-pair-iso}
	\vcenter{\xymatrix{
	\fD(\fs) \ar@{->>}[r]^{\rho_V} \ar@{->>}[dr]_{\rho_{V^\ve}} & \fs_{V} \ar@{->}[d]^{\cong} \\
	 & \fs_{V^\ve}.
	}}
	\end{equation}

\subsubsection{Self-dual Lie algebra representations}

Let $V = S^\lambda \in \Irr(G)$ such that $V \cong V^\ve$, i.e., such that $\lambda = \lambda'$. Let $U = S_\R^\lambda$ be a realization of $V$ over $\R$, so also $U^\ve = U \otimes_\R \ve \cong U$, and let $\phi : U \to U$ be an associator as in Section \ref{subsec:generalities}. Then $\phi^2 = \alpha_U \cdot \id_{U}$ with $\alpha_U \in \set{\pm 1}$, so $\phi^{-1} = \alpha_U \cdot \phi$.

Let $\subgrp{-,-}$ be a $G$-invariant (symmetric, positive definite) inner product on $U$, and let $\phi^*: U \to U$ be the adjoint of $\phi$ with respect to this inner product, i.e., $\phi^*$ is the unique linear map such that for all $u,v \in U$ one has $\subgrp{\phi(u),v} = \subgrp{u,\phi^*(v)}$. By the $G$-invariance of the inner product, it follows that $\phi^*$ also satisfies \eqref{eq:associator}. Then $\phi^* \circ \phi \in \Hom_{\RG}(U,U)$, so Schur's Lemma and \cite[Lemma 7.4]{Lam:2001} imply that $\phi^* \circ \phi = \beta_U \cdot \id_U$ for some $0 \neq \beta_U \in \R$. Further, for $0 \neq u \in U$ one has
	\[
	0 < \subgrp{\phi(u),\phi(u)} = \subgrp{u,(\phi^* \circ \phi)(u)} = \beta_U \cdot \subgrp{u,u},
	\]
and $\subgrp{u,u} > 0$, so $\beta_U > 0$. Now $\alpha_U \cdot \phi = \phi^{-1} = (1/\beta_U) \cdot \phi^*$, so $\phi^* = (\alpha_U \beta_U) \cdot \phi$.

Now define a new bilinear form $(-|-)'$ on $U$ by $(u|v)' = \subgrp{\phi(u),v}$. This form is non-degenerate because $\phi$ is a linear isomorphism. For $u,v \in U$ one has
	\begin{align*}
	(u|v)' &= \subgrp{\phi(u),v} \\
	&= \subgrp{u,\phi^*(v)} \\
	&= (\alpha_U \beta_U) \cdot \subgrp{u,\phi(v)} & \text{because $\phi^* = (\alpha_{U} \beta_{U}) \cdot \phi$,} \\
	&= (\alpha_U \beta_U) \cdot \subgrp{\phi(v),u} & \text{because $\subgrp{-,-}$ is symmetric,} \\
	&= (\alpha_U \beta_U) \cdot (v|u)'.
	\end{align*}
But then $(u|v)' = (\alpha_U \beta_U) \cdot (v|u)' = (\alpha_U \beta_U)^2 \cdot (u|v)'$. Since the form $(-|-)'$ is non-degenerate, this implies that $(\alpha_U \beta_U)^2 = 1$. Since $\beta_U > 0$ and $\alpha_U \in \set{\pm 1}$, we deduce that $\beta_U = 1$ and $(u|v)' = \alpha_U \cdot (v|u)'$. Thus the form $(-|-)'$ is either symmetric or anti-symmetric. Finally, extending scalars bilinearly from $\R$ to $\C$, the form $(-|-)'$ lifts to a non-degenerate bilinear form $(-|-)$ on $V$ that is symmetric if $\alpha_U = +1$ and is anti-symmetric if $\alpha_U = -1$.

\begin{lemma} \label{lemma:osp-factorization}
Let $V \in \Irr(G)$ such that $V \cong V^\ve$, and let $(-|-)$ be the non-degenerate bilinear form on $V$ defined as above via the choice of a real form $U$ for $V$. Then $\rho_V(\fs) \subseteq \osp(V)$, where
	\[
	\osp(V) := \{ x \in \gl(V) : (x.u|v) + (u|x.v) = 0 \text{ for all } u,v \in V \} = \begin{cases} 
	\fso(V) & \text{if $\alpha_U = +1$,} \\
	\fsp(V) & \text{if $\alpha_U = -1$.}
	\end{cases}
	\]
\end{lemma}

\begin{proof}
If $u,v \in U$ and $s \in S$, then
	\[
	(s.u|v)' = \subgrp{\phi(s.u),v} = -\subgrp{s.\phi(u),v} = - \subgrp{\phi(u),s^{-1}.v} = - \subgrp{\phi(u),s.v} = -(u|s.v)'.
	\]
Extending scalars from $\R$ to $\C$, this implies first for all $s \in S$, and then for all $x \in \fs$ by linearity, that $(x.u|v) + (u|x.v) = 0$ for all $u,v \in V$.
\end{proof}

\begin{remark} \label{rem:Ve-V-even-dim}
If $V \in \Irr(G)$ and $V \cong V^\ve$, then $\dim(V)$ is even, e.g., by Lemma \ref{lemma:classical-Clifford}\eqref{item:type-F-simple}. Say $\dim(V) = 2r$. Then $\osp(V)$ is a simple Lie algebra of type $C_r$ and dimension $2r^2 +r$ if $\alpha_U = -1$, and is a simple Lie algebra of type $D_r$ (and dimension $2r^2-r$, for $r \geq 3$) if $\alpha_U = 1$.
\end{remark}

The next lemma is a partial converse to the construction in this section:

\begin{lemma} \label{lemma:form-implies-dual}
Let $V \in \Irr(G)$ with $\dim(V) > 1$, and suppose the set of reflections in $G$ consists of a single conjugacy class. If there exists a non-degenerate $\fD(\fs)$-invariant bilinear form on $V$, then $V \cong V^\ve$ as $\CG$-modules.
\end{lemma}

\begin{proof}
Suppose there exists a non-degenerate $\fD(\fs)$-invariant bilinear form $(-|-)$ on $V$. Thus for all $x \in \fD(\fs)$ and all $u,v \in V$, one has $(x.u|v) = -(u|x.v)$. Then the map $v \mapsto (v|-)$ defines a $\fD(\fs)$-module isomorphism $V \cong V^{*,\Lie}$. Since $V^{*,\Lie} \cong V^\ve$ as $\fs$-modules by Corollary \ref{cor:tensor-sign-restriction}, we have $V \cong V^\ve$ as $\fD(\fs)$-modules, and hence $V \cong V^\ve$ as $\CG$-modules by Proposition \ref{prop:iso-equivalent}.
\end{proof}

\subsubsection{Exterior powers of the reflection representation} \label{sububsec:reflection-rep}

Let $V$ be a $\CG$-module, and let
	\[
	\rho_{\Grp}^{\wedge^d}: \fs \to \End(\Lambda^d(V)) \qquad \text{and} \qquad \rho_{\Lie}^{\wedge^d}: \fs \to \End(\Lambda^d(V))
	\]
be the $\fs$-module structure maps for $\Res^{\CG}_{\fs}(\Lambda^d(V))$ and $\Lambda^d(\Res^{\CG}_{\fs}(V))$, respectively. Then for $s \in S$ and $v_1,v_2,\ldots,v_d \in V$, one has
	\begin{gather}
	\rho_{\Grp}^{\wedge^d}(s)(v_1 \wedge v_2 \wedge \cdots \wedge v_d) = (s.v_1) \wedge (s.v_2) \wedge \cdots \wedge (s.v_d), \label{eq:rho-grp-action} \\
	\textstyle \rho_{\Lie}^{\wedge^d}(s)(v_1 \wedge v_2 \wedge \cdots \wedge v_d) = \sum_{i=1}^d v_1 \wedge \cdots \wedge (s.v_i) \wedge \cdots \wedge v_d.
	\end{gather}
Of course, \eqref{eq:rho-grp-action} is true as written for any $s \in G$, not just for $s \in S$.

Now recall the reflection representation $V = V_{\refl}$ of $G$. Let $(G,S_0)$ be a Coxeter presentation for $G$ as in Section \ref{subsec:generalities}, with $S_0 = \{s_1,\ldots,s_n\}$. Then $V$ admits a basis $\{\alpha_1,\ldots,\alpha_n\}$ on which the simple reflections act by $s_i .\alpha_j = \alpha_j + 2 \cdot \cos(\pi / m_{i,j} ) \cdot \alpha_i$. Set $b(\alpha_i,\alpha_j) = -\cos(\pi/m_{i,j})$. Then $b$ extends bilinearly to a symmetric bilinear form $b(-,-)$ on $V$, and the action of $s_i$ on an arbitrary vector $v \in V$ is given by $s_i . v = v - 2 \cdot b(\alpha_i,v) \alpha_i$.

\begin{theorem}[{\cite[5.1.4]{Geck:2000}}] \label{thm:exterior-powers-reflection-rep}
Let $(G,S_0)$ be an irreducible finite Coxeter system of rank $n = \abs{S_0}$, and let $V$ be the corresponding reflection representation of $G$. For $0 \leq d \leq n$, the $\CG$-modules $\Lambda^d(V)$ are irreducible and pairwise non-isomorphic.
\end{theorem}

In the special case $V = V_{\refl}$, the group-theoretic and the Lie-theoretic actions of $\fD(\fs)$ on the exterior powers $\Lambda^d(V)$ are the same:

\begin{lemma} \label{lemma:ext-powers-refl-rep}
Let $(G,S_0)$ be an irreducible Coxeter system of rank $n = \abs{S_0}$, and let $V = V_{\refl}$ be the reflection representation of $G$. Then for $0 \leq d \leq n$, one has
	\[
	\Res^{\CG}_{\fD(\fs)}(\Lambda^d(V)) = \Lambda^d(\Res^{\CG}_{\fD(\fs)}(V)).
	\]
\end{lemma}

\begin{proof}
The argument is a generalization of the proof of \cite[Lemme 12]{Marin:2007}. Fix $0 \leq d \leq n$, and let
	\[
	\rho_{\Grp} = \rho_{\Grp}^{\wedge^d}: \fs \to \End(\Lambda^d(V)) \qquad \text{and} \qquad \rho_{\Lie} = \rho_{\Lie}^{\wedge^d}: \fs \to \End(\Lambda^d(V))
	\]
be the group- and Lie-theoretic module structure maps for the action of $\fs$ on $\Lambda^d(V)$. We will show for all $s \in S_0$ that $\rho_{\Lie}(s) = \rho_{\Grp}(s) + (d-1) \cdot \Id_{\Lambda^d(V)}$. Since the simple reflections generate $\fs$ by Remark \ref{remark:s-generated-by-S0}, and since $(d-1) \cdot \Id_{\Lambda^d(V)}$ is a central operator in $\End(\Lambda^d(V))$, this will imply that the restrictions of $\rho_{\Grp}$ and $\rho_{\Lie}$ to the derived algebra $\fD(\fs)$ are equal.

First let $s_k \in S_0$ be a simple reflection such that the corresponding vertex in the Coxeter diagram of $G$ is adjacent to one other vertex. Let $\ul{\alpha} = \alpha_{i_1} \wedge \cdots \wedge \alpha_{i_d} \in \Lambda^d(V)$ for distinct indices $1 \leq i_1,\ldots,i_d \leq n$. We consider the values $\rho_{\Grp}(s_k)(\ul{\alpha})$ and $\rho_{\Lie}(s_k)(\ul{\alpha})$. Reordering the factors in $\ul{\alpha}$ if necessary, there are three basic cases to consider:
	\begin{enumerate}
	\item $b(\alpha_k,\alpha_{i_j}) = 0$ for all $1 \leq j \leq d$. 

	\item $b(\alpha_k,\alpha_{i_1}) \neq 0$ and $b(\alpha_k,\alpha_{i_j}) = 0$ for all $2 \leq j \leq d$. 

	\item $i_1 = k$, $b(\alpha_k,\alpha_{i_2}) \neq 0$, and $b(\alpha_k,\alpha_{i_j}) = 0$ for all $3 \leq j \leq d$. 
	\end{enumerate}
In each case it is straightforward to check that $\rho_{\Lie}(s_k)(\ul{\alpha}) = \rho_{\Grp}(s_k)(\ul{\alpha}) + (d-1) \cdot \ul{\alpha}$, and hence $\rho_{\Lie}(s_k) = \rho_{\Grp}(s_k) + (d-1) \cdot \Id_{\Lambda^d(V)}$ as operators on $\Lambda^d(V)$. Now for an arbitrary $s_i \in S_0$, we can first choose a simple reflection $s_k$ as above and an element $g \in G$ such that $g s_k g^{-1} = s_i$; this is possible because all (simple) reflections of the same root length are conjugate in $G$. Then
	\begin{align*}
	\rho_{\Lie}(s_i) &= \rho_{\Grp}(g) \circ \rho_{\Lie}(s_k) \circ \rho_{\Grp}(g^{-1}) \\
	&= \rho_{\Grp}(g) \circ \big( \rho_{\Grp}(s_k) + (d-1) \cdot \Id \big) \circ \rho_{\Grp}(g^{-1}) \\
	&= \rho_{\Grp}(s_i) + (d-1) \cdot \Id. \qedhere
	\end{align*}
\end{proof}

Let $\Delta_{d}: \fsl(V) \to \fsl(\Lambda^d(V))$ be the Lie algebra homomorphism corresponding to the action by derivations of $\fsl(V)$ on $\Lambda^d(V)$.

\begin{corollary} \label{cor:exterior-powers-factorization}
Let $(G,S_0)$ be an irreducible Coxeter system of rank $n = \abs{S_0}$, and let $V$ be the corresponding reflection representation of $G$. Then for $0 \leq d \leq n$, the module structure map $\rho_{\Lambda^d(V)}: \fD(\fs) \to \fsl(\Lambda^d(V))$ admits the Lie algebra factorization
	\begin{equation} \label{eq:linear-factorizations}
	\vcenter{\xymatrix{
	\fD(\fs) \ar@{->>}[r]^{\rho_V} \ar@{->}[drr]_{\rho_{\Lambda^d(V)}} & \fs_V \ar@{->}[r]^{\subseteq} & \fsl(V) \ar@{->}[d]^{\Delta_d} \\
	& & \fsl(\Lambda^d(V))
	}}.
	\end{equation}
\end{corollary}

\begin{remark} \label{rem:Delta-d-image}
Let the notation and assumptions be as in Corollary \ref{cor:exterior-powers-factorization}, and suppose $\fs_V = \fsl(V)$. Then for $0 < d < n$, the simplicity of $\fsl(V)$ as a Lie algebra implies that
	\[
	\fs_V = \fsl(V) \cong \im(\Delta_d) = \im(\rho_{\Lambda^d(V)}).
	\]
On the other hand, if $d \in \set{0,n}$, then $\dim(\Lambda^d(V)) = 1$, and hence
	\[
	\fs_{\Lambda^d(V)} = \im(\rho_{\Lambda^d(V)}) = \im(\Delta_d) = \fsl(\Lambda^d(V)) = 0.
	\]
\end{remark}

\subsubsection{Summary}\label{SS:summary-and-injection}

Let $V_{\refl}$ be the reflection representation of $G$ as in Section \ref{sububsec:reflection-rep}, and let
	\[
	\Irr^{> 1}(G) = \set{ V \in \Irr(G) : \dim(V) > 1}.
	\]
Define subsets $E$ and $F$ of $\Irr^{> 1}(G)$ by
	\begin{equation} \label{eq:E-F-Irr} 
	\begin{split}
	E = E(G) &= \{ V \in \Irr^{> 1}(G): V \not\cong V^\ve \text{ and } V \not\cong \Lambda^d(V_{\refl}) \text{ for all $d$} \}, \\
	F = F(G) &= \{ V \in \Irr^{> 1}(G) : V \cong V^\ve \text{ and } V \not\cong \Lambda^d(V_{\refl}) \text{ for all $d$} \}. 
	\end{split}
	\end{equation}
We may refer to the elements of these sets as simple modules of types $E$ and $F$, respectively. Let $\sim$ be the equivalence relation on $\Irr(G)$ generated by $V \sim V^\ve$. Then it follows from \eqref{eq:dual-pair-iso}, Lemma \ref{lemma:osp-factorization}, and Corollary \ref{cor:exterior-powers-factorization} that the Lie algebra homomorphism \eqref{eq:D(s)-Artin-Wedderburn} restricts further to an injective Lie algebra map
	\begin{equation} \label{eq:D(s)-factorization}
	\fD(\fs) \hookrightarrow \fsl(V_{\refl}) \oplus \Big[ \bigoplus_{V \in E/\sim} \fsl(V) \Big] \oplus \Big[ \bigoplus_{V \in F} \osp(V) \Big],
	\end{equation}
where the sum over $E/\!\!\sim$ means that we take the sum over any complete set of representatives in $E$ for the distinct equivalence classes under the relation $\sim$. In some cases \eqref{eq:D(s)-factorization} is an isomorphism, but in general the image of $\fD(\fs)$ is strictly smaller.

\subsection{Simple ideals}

Given $V \in \Irr(G)$ with module structure map $\rho_V : \CG \to \End(V)$, let $\fs^V$ be the orthogonal complement in $\fD(\fs)$ of $\ker(\rho_V|_{\fD(\fs)})$ with respect to the Killing form on $\fD(\fs)$. Then $\fs^V$ is an ideal in $\fD(\fs)$ such that $\fD(\fs) = \fs^V \oplus \ker(\rho_V|_{\fD(\fs)})$ as Lie algebras, and $\rho_V$ induces a Lie algebra isomorphism $\fs^V \cong \rho_V(\fD(\fs)) = \fs_V$. The next lemma extrapolates arguments from the proof of \cite[Proposition 5.10.1]{DK:2026a}, which in turn was based on the arguments in \cite[\S6.7]{Marin:2007}.

\begin{lemma} \label{lemma:distinct-simple-ideals}
Let $E$ and $F$ be the subsets of $\Irr(G)$ defined in \eqref{eq:E-F-Irr}, and let $X$ be a subset of $\Irr^{> 1}(G)$ with the following properties:
	\begin{enumerate}
	\item \label{item:V-then-no-Ve} If $V \in X \cap E$, then $V^\ve \notin X$, i.e., $X$ does not contain the isomorphism class of $V^\ve$.
	\item If $V \in X \cap E$, then $\fs_V = \fsl(V)$ and $\dim(V) \geq 3$.
	\item If $V \in X \cap F$, then $\fs_V \cong \osp(V)$ and $\dim(V) \geq 8$.
	\item \label{item:at-most-1-dim8} There is at most one $V \in X \cap F$ such that $\dim(V) = 8$.
	\end{enumerate}
Suppose also that the set $S$ of reflections in $G$ consists of a single conjugacy class in $G$. Then the set $\smallset{ \fs^V : V \in X}$ consists of pairwise distinct simple ideals in $\fD(\fs)$, and $\sum_{V \in X} \fs^V = \bigoplus_{V \in X} \fs^V$.
\end{lemma}

\begin{proof}
Let $U,V \in X$, and suppose $\fs^U = \fs^V$. By the assumptions and Remark \ref{rem:Ve-V-even-dim}, $\fs^U$ and $\fs^V$ are simple Lie algebras of type $A_r$ with $r \geq 2$, or of type $C_r$ or $D_r$ with $r \geq 4$. These types are pairwise non-isomorphic, so $\fs^U$ and $\fs^V$ must be of the same Lie type.

First suppose $\fs^U$ and $\fs^V$ are of type $A_r$ with $r \geq 2$ (so $U,V \in X \cap E$). Then fixing bases for $U$ and $V$, one has $\fs_U = \fsl(\C^{r+1}) = \fs_V$, and the composite induced map
	\[
	\fsl(\C^{r+1}) = \fs_U \xrightarrow{\rho_U^{-1}} \fs^U = \fs^V \xrightarrow{\rho_V} \fs_V = \fsl(\C^{r+1})
	\]
is a Lie algebra automorphism. Lie algebra automorphisms of $\fsl(\C^{r+1})$ come in two forms:
	\begin{itemize}
	\item $X \mapsto gXg^{-1}$ for some $g \in GL(\C^{r+1})$, or
	\item $X \mapsto -(gX^t g^{-1})$ for some $g \in GL(\C^{r+1})$, where $X^t$ denotes the transpose of $X$;
	\end{itemize}
see \cite[IX.5]{Jacobson:1979}. If $\rho_V \circ \rho_U^{-1}$ is of the first form, then $U \cong V$ as $\fD(\fs)$-modules, and hence $U \cong V$ as $\CG$-modules, by Proposition \ref{prop:iso-equivalent}. Then $U = V$, because $\Irr(G)$ contains a single representative from each isomorphism class of simple $\CG$-modules. On the other hand, if $\rho_V \circ \rho_U^{-1}$ is of the second form, then $U \cong V^{*,\Lie}$ as $\fD(\fs)$-modules. Then $U \cong V^\ve$ as $\fD(\fs)$-modules, by Corollary \ref{cor:tensor-sign-restriction}, and hence $U \cong V^\ve$ by Proposition \ref{prop:iso-equivalent}. This contradicts assumption \eqref{item:V-then-no-Ve}, so we conclude that $\rho_V \circ \rho_U^{-1}$ must be of the first form, and $U = V$.

Now suppose $\fs^U = \fs^V$ is of type $C_r$ or $D_r$ with $r \geq 4$. If the ideal is of type $D_4$, then assumption \eqref{item:at-most-1-dim8} immediately implies that $U = V$. Now excluding type $D_4$, a simple Lie algebra of type $C_r$ with $r \geq 4$, or type $D_r$ with $r \geq 5$, has (up to isomorphism) a unique irreducible representation of dimension $2r$. This implies that $U \cong V$ as $\fD(\fs)$-modules, and hence $U \cong V$ as $\CG$-modules by Proposition \ref{prop:iso-equivalent}. Then $U = V$.

Thus, the set $\smallset{ \fs^V : V \in X}$ consists of pairwise distinct simple ideals in $\fD(\fs)$. This implies that the sum $\sum_{V \in X} \fs^V$ is a direct sum in $\fD(\fs)$.
\end{proof}

The last lemma of this section records a variation of \cite[Lemme 15]{Marin:2007}.

\begin{lemma} \label{lemma:simple-ideal-criterion}
Let $V$ be a finite-dimensional $\C$-vector space, and let $\fh \subset \fg$ be two semisimple Lie subalgebras of $\fsl(V)$ such that $\g$ acts irreducibly on $V$. Suppose also that there exist simple factors $\fh_1,\fh_2,\ldots,\fh_s$ in $\fh$ such that:
	\begin{enumerate}
	\item for each $j$, there exists an $\fh_j$-module composition factor $W_j$ in $V$ of multiplicity one;
	\item for each $j$, one has $\fh_j \cong \fsl(W_j)$; and
	\item $\big( \sum_{j=1}^s \dim(W_j) \big)^2 > \dim(V)$.
	\end{enumerate}
Then $\g$ is a simple Lie algebra.
\end{lemma}

\begin{proof}
Let $\g = \g_1 \oplus \cdots \oplus \g_t$ be the decomposition of the semisimple Lie algebra $\g$ into simple ideals. We want to show that $t = 1$. Since $\g$ acts irreducibly on $V$, it follows for each $1 \leq i \leq t$ that there exists a simple $\g_i$-module $V_i$ such that $V$ identifies with the external tensor product of modules $V_1 \boxtimes \cdots \boxtimes V_t$. Since $\g \subset \fsl(V)$ acts faithfully on $V$, it follows for each $i$ that $V_i$ cannot be the trivial $\g_i$-module (else $\g_i$ acts trivially on $V$), and hence $\dim(V_i) \geq 2$.

Fix some $1 \leq i \leq t$. Then for each $j$, the space $V_i$ becomes an $\fh_j$-module via the sequence of canonical Lie algebra maps $\fh_j \subseteq \fh \subseteq \g \twoheadrightarrow \g_i$. If $V_i$ were trivial as an $\fh_j$-module, it would follow that each $\fh_j$-composition factor in $V$ would be of multiplicity at least $\dim(V_i) \geq 2$, contradicting the assumption that $\fh_j$ has a composition factor $W_j$ in $V$ of multiplicity one. Then $V_i$ must be nontrivial as an $\fh_j$-module. Now consider $V_i$ as a module for the semisimple Lie algebra $\fh' := \fh_1 \oplus \cdots \oplus \fh_s$. The simple $\fh'$-module factors in $V_i$ are of the form $W = W_{i,1} \boxtimes \cdots \boxtimes W_{i,s}$, where $W_{i,j}$ is a simple $\fh_j$-module. For such $W$, one has
	\begin{equation} \label{eq:prod-geq-sum}
	\dim(W) = \prod_{1 \leq j \leq s} \dim(W_{i,j}) = \prod_{\substack{1 \leq j \leq s \\ \dim(W_{i,j}) > 1}} \dim(W_{i,j}) \geq \sum_{\substack{1 \leq j \leq s \\ \dim(W_{i,j}) > 1}} \dim(W_{i,j}).
	\end{equation}
By assumption, $\fh_j \cong \fsl(W_j)$, and the Lie algebra $\fsl(k)$ has no nontrivial simple modules of dimension less than $k$, so if $\dim(W_{i,j}) > 1$, then $\dim(W_{i,j}) \geq \dim(W_j)$. Also, for each $j$, there must be some simple $\fh'$-module factor $W$ in $V_i$ such that $\dim(W_{i,j}) > 1$, because $V_i$ is nontrivial as an $\fh_j$-module. Together with \eqref{eq:prod-geq-sum}, this implies that $\dim(V_i) \geq \sum_{j=1}^s \dim(W_j)$.

Now if $t \geq 2$, one would have
	\[ \textstyle
	\dim(V) \geq \dim(V_1) \cdot \dim(V_2) \geq \big( \sum_{j=1}^s \dim(W_j) \big)^2 > \dim(V).
	\]
This is a contradiction, so we conclude that $t = 1$, and hence that $\g$ is simple.
\end{proof}

\section{Lie algebra calculations for groups of exceptional type} \label{S:Lie-algebra-calcs}

\subsection{Type \texorpdfstring{$I_2(m)$ for $m \geq 3$}{I2(m) for m >=3}} \label{subsec:I2m-LA}

The reflection group of type $I_2(m)$ is the dihedral group of order $2m$:
	\begin{align*}
	D_m = \subgrp{r,s : r^m = s^2 = (sr)^2 = 1} = \smallset{ 1, r, r^2, \ldots, r^{m-1}, s, sr, sr^2, \ldots, sr^{m-1}}.
	\end{align*}
The set of reflections in $D_m$ is $S = \smallset{s,sr,\ldots,sr^{m-1}}$, and a set of simple reflections is $S_0 = \set{s,sr}$. If $m$ is odd then $S$ consists of a single conjugacy class, while if $m$ is even then $S$ consists of the two conjugacy classes $C_\ell = \smallset{s,sr^2,\ldots,sr^{m-2}}$ and $C_s = \smallset{sr,sr^3,\ldots,sr^{m-1}}$.

\begin{theorem} \label{thm:D(s(Dm)-LA}
For the dihedral group $D_m$ of order $2m$ ($m \geq 3$), the injective Lie algebra map
	\begin{equation} \label{eq:Dm-AW-isomorphism}
	\fD(\fs) \hookrightarrow \bigoplus_{V \in \Irr(D_m)} \fsl(V)
	\end{equation}
of \eqref{eq:D(s)-Artin-Wedderburn} is an isomorphism.
\end{theorem}

\begin{proof}
The proof is by dimension comparison. The details depend on the parity of $m$. 

First suppose $m = 2k$ is even. Then $\Irr(D_m)$ consists of 4 one-dimensional modules and $k-1$ two-dimensional modules; see \cite[\S2.4]{DK:2025} for explicit descriptions. By the injectivity of \eqref{eq:Dm-AW-isomorphism}, this implies that $\dim(\fD(\fs)) \leq 3 \cdot (k-1) = 3k-3$. Then by Lemma \ref{lemma:g-reductive}, $\dim(\fs) \leq 3k-1$. To see that $\dim(\fs) = 3k-1$, and hence $\dim(\fD(\fs)) = 3k-3$, observe first that $\fs$ contains the $m$ reflections $s,sr,\ldots,sr^{m-1}$. Taking Lie brackets, $\fs$ also contains the elements $[s,sr^i] = r^i - r^{m-i}$ for $1 \leq i \leq k-1$. Then $\dim(\fs) \geq m + (k-1) = 3k-1$, and hence $\dim(\fs) = 3k-1$.

Now suppose $m = 2k+1$ is odd. In this case $\Irr(D_m)$ consists of $2$ one-dimensional modules and $k$ two-dimensional modules. Then as above one has $\dim(\fD(\fs)) \leq 3 \cdot k$, and hence $\dim(\fs) \leq 3k+1$. Since $\fs$ contains the $m$ reflections as well as the brackets $[s,sr^i] = r^i - r^{m-i}$ for $1 \leq i \leq k$, one sees that $\dim(\fs) = 3k+1$, and hence $\dim(\fD(\fs)) = 3k$.
\end{proof}

\begin{remark}
One can show for each simple $\CD_m$-module $V$ with $\dim(V) = 2$ that $V \cong V^\ve$. Since $\fsp(2) = \fsl(2)$, Theorem \ref{thm:D(s(Dm)-LA} thus asserts for $G = D_m$ that the injective Lie algebra map \eqref{eq:D(s)-factorization} is an isomorphism.
\end{remark}

\subsection{Type \texorpdfstring{$H_3$}{H3}}

The reflection group of type $H_3$ is the symmetry group of the regular dodecahedron, or dually, of the regular icosahedron.

\begin{theorem} \label{thm:D(H3)-LA}
For the reflection group of type $H_3$, the injective Lie algebra map \eqref{eq:D(s)-factorization} is an isomorphism. Specifically,
	\begin{equation} \label{eq:D(H3)-LA}
	\fD(\fs) \cong \fsl(V_{\refl}) \oplus \bigoplus_{V \in E/\sim} \fsl(V) \cong \fsl(3) \oplus \fsl(3) \oplus \fsl(4) \oplus \fsl(5).
	\end{equation}
\end{theorem}

\begin{proof}
Let $G = W(H_3)$ be the Coxeter group of type $H_3$. A list of the simple $\CG$-modules and their dimensions is given in Table \ref{tab:H3}. Via this labeling, the reflection representation of $G$ is $U_5$, and its exterior powers are $\Lambda^0(U_5) = U_1$, $\Lambda^2(U_5) = U_3$, and $\Lambda^3(U_5) = U_2$. All remaining simple modules are of type $E$. Then $\dim(\fD(\fs)) \leq 55$ by \eqref{eq:D(s)-factorization}. On the other hand, using GAP \cite{GAP4} we have calculated that $\dim(\fs) = 56$. Since $Z(\fs)$ is one-dimensional, spanned by $\calC_\ell$, this implies that $\dim(\fD(\fs)) = 55$, and hence \eqref{eq:D(s)-factorization} is an isomorphism.
\end{proof}

\begin{remark} \label{remark:D(H3)-LA}
With notation as in the preceding proof, one has $(U_5)^\ve \cong U_3$. Then \eqref{eq:D(H3)-LA} can be rewritten as
	\begin{equation} \label{eq:D(H3)-LA-compact}
	\fD(\fs) \cong \bigoplus_{V \in \Irr^{>1}(G)/\sim} \fsl(V).
	\end{equation}
\end{remark}

\subsection{Type \texorpdfstring{$H_4$}{H4}}

The reflection group $G = W(H_4)$ of type $H_4$ is the full symmetry group of a $120$-sided solid (with dodecahedral faces) in $\R^4$, or dually, of a regular $600$-sided solid (with tetrahedral faces). In this case the injective Lie algebra map \eqref{eq:D(s)-factorization} admits a further proper factorization.

Recall that $G$ splits over $\Q(\sqrt{5})$; see \cite[5.3.8]{Geck:2000}. Let $\sigma : \C \to \C$ be a field automorphism lifting the nontrivial automorphism of $\Q(\sqrt{5})$. Given a $\CG$-module $V$, let ${}^\sigma V = \C \otimes_{\sigma} V$ be the $\CG$-module obtained from $V$ by `scalar extension' along $\sigma$.\footnote{Writing ${}^\sigma v = 1 \otimes_\sigma v$, one has $g.({}^\sigma v) = {}^\sigma(g.v)$ for $g \in G$, and $\lambda \cdot {}^\sigma v = {}^\sigma( \sigma^{-1}(\lambda) \cdot v)$ for $\lambda \in \C$.} The next lemma can then be verified in the same manner as Lemma \ref{lemma:ext-powers-refl-rep}; the only difference in the proof is that the values of the bilinear form $b$ are replaced by their conjugates under $\sigma$.

\begin{lemma} \label{lemma:ext-powers-sigma-twist}
Let $G = W(H_4)$, and let $V = V_{\refl}$ be the $4$-dimensional reflection representation of $G$. Then for $0 \leq d \leq 4$, one has
	\[
	\Res^{\CG}_{\fD(\fs)}(\Lambda^d({}^\sigma V)) =  \Lambda^d(\Res^{\CG}_{\fD(\fs)}({}^\sigma V)),
	\]
and the evident analogues of Corollary \ref{cor:exterior-powers-factorization} and Remark \ref{rem:Delta-d-image} hold with $V$ replaced by ${}^\sigma V$.
\end{lemma}

\begin{lemma} \label{lemma:H4-V9-V10-factorization}
Let $G = W(H_4)$, and let the simple $\CG$-modules be labeled as in Table \ref{tab:H4}. Then
	\[
	\dim(\fs_{V_9}) = \dim(\fs_{V_{10}}) = \dim(\im(\rho_{V_9} \oplus \rho_{V_{10}})) = 28.
	\]
Consequently, there exists a commutative diagram of Lie algebras
	\[
	\vcenter{\xymatrix{
	\fD(\fs) \ar@{->>}[r]^{\rho_{V_9}} \ar@{->>}[dr]_{\rho_{{V_{10}}}} & \fs_{V_9} \ar@{->}[d]^{\cong} \\
	 & \fs_{V_{10}}.
	}}
	\]
\end{lemma}

\begin{proof}
By calculation in GAP \cite{GAP4}.
\end{proof}

Recall the subsets $E$ and $F$ of $\Irr(G)$ defined in \eqref{eq:E-F-Irr}, and note that ${}^\sigma V_{\refl} \in E$ (see Table \ref{tab:H4}). Define subsets $E'$ and $F'$ of $\Irr(G)$ by
	\begin{equation} \label{eq:E'-F'-Irr}
	\begin{split}
	E' &= \{ V \in E: V \not\cong \Lambda^d({}^\sigma V_{\refl}) \text{ for all $d$} \}, \\
	F' &= \{ V \in F : V \not\cong \Lambda^d({}^\sigma V_{\refl}) \text{ for all $d$} \},
	\end{split}
	\end{equation}
and let $F''$ be the subset of $F'$ obtained by omitting (either) one of the two $8$-dimensional simple $\CG$-modules. Then Lemmas \ref{lemma:ext-powers-sigma-twist} and \ref{lemma:H4-V9-V10-factorization} imply that \eqref{eq:D(s)-factorization} restricts further to an injective Lie algebra homomorphism
	\begin{equation} \label{eq:H4-D(s)-factorization}
	\fD(\fs) \hookrightarrow \fsl(V_{\refl}) \oplus \fsl({}^\sigma V_{\refl}) \oplus \Big[ \bigoplus_{V \in E'/\sim} \fsl(V) \Big] \oplus \Big[ \bigoplus_{V \in F''} \osp(V) \Big].
	\end{equation}

\begin{theorem} \label{thm:H4-D(s)-LA}
With notation and assumptions as above, the injective Lie algebra map \eqref{eq:H4-D(s)-factorization} is an isomorphism, and $\osp(V) = \fso(V)$ for each $V \in F''$. Specifically,
	\begin{gather}
	\fD(\fs) \cong \fsl(V_{\refl}) \oplus \fsl({}^\sigma V_{\refl}) \oplus \Big[ \bigoplus_{V \in E'/\sim} \fsl(V) \Big] \oplus \Big[ \bigoplus_{V \in F''} \fso(V) \Big] \label{eq:D(s)-H4} \\
	\cong \fsl(4)^{\times 2} \oplus \fso(8) \oplus \fsl(9)^{\times 2} \oplus \fso(10) \oplus \fsl(16)^{\times 2} \oplus \fso(16)^{\times 2} \oplus \fso(18) \notag \\
	{} \oplus \fso(24)^{\times 4} \oplus \fsl(25) \oplus \fso(30)^{\times 2} \oplus \fsl(36) \oplus \fso(40) \oplus \fso(48). \notag
	\end{gather}
\end{theorem}

\begin{proof}
Let $X = \set{V_{\refl},{}^\sigma V_{\refl}} \cup (E'/\!\!\sim) \cup F''$. Here, $E'/\!\!\sim$ denotes a complete set of representatives for the equivalence classes in $E'$ under the relation $\sim$. The first main step of the proof is to show for all $V \in X$ that $\fs_V = \rho_V(\fD(\fs))$ is a simple Lie algebra. Let $S_0 = \set{s_1,s_2,s_3,s_4}$ be a set of Coxeter generators for $G$ such that the corresponding Coxeter graph is labeled as in \cite[Theorem 1.3.3]{Geck:2000}, and let $H \cong W(H_3)$ be the subgroup of $G$ generated by the set $\set{s_1,s_2,s_3}$. We will show that $\fs_V$ is simple by considering $\Res^G_H(V)$, the restriction of $V$ to the subgroup $H$.

Let $\fh \subset \fs$ be the Lie subalgebra of $\Lie(\CH) \subset \Lie(\CG)$ generated by the set of reflections in $H$. Given $V \in \Irr(G)$ with module structure map $\rho_V : \CG \to \End(V)$, set $\fh_V = \rho_V(\fD(\fh))$. Then $\fh_V \subseteq \fs_V$ are semisimple Lie subalgebras of $\fsl(V)$, and $\fs_V$ acts irreducibly on $V$, by Lemma \ref{lemma:simples-remain-simple}. For $W \in \Irr(H)$, set $\fh_W = \rho_W(\fD(\fh))$. It follows from Theorem \ref{thm:D(H3)-LA} (as re-expressed in Remark \ref{remark:D(H3)-LA}) and Remark \ref{rem:Delta-d-image} that $\fh_W \cong \fsl(W)$ for each $W \in \Irr(H)$, and
	\[
	\fh_V \cong \bigoplus_{\substack{W \in \Irr^{>1}(H)/\sim \\ W \leq V \text{ or } W^{\ve} \leq V}} \fsl(W).
	\]
In other words, $\fh_V$ consists of a copy of $\fsl(W)$ for each $W \in \Irr^{>1}(H)/\!\!\sim$ such that either $W$ or $W^{\ve}$ occurs with multiplicity at least one in $\Res^G_H(V)$.

Now fix some $V \in X$. Assume for the moment that $V$ is not the $48$-dimensional simple module, labeled $V_{34}$ in Table \ref{tab:H4}; we will treat that case separately via a modified argument in Lemma \ref{lemma:V34} below. The composition factor multiplicities of $\Res^G_H(V)$ can be explicitly determined, for example by character calculations using GAP \cite{GAP4}. By inspection, one observes (for $V \neq V_{34}$) that there exist $\CH$-modules $W_1,\ldots,W_s \in \Irr^{>1}(H)/\!\!\sim$ such that the following conditions hold:
	\begin{itemize}
	\item For each $1 \leq j \leq s$, either $W_j$ or $(W_j)^\ve$ occurs with multiplicity one in $\Res^G_H(V)$.
	\item $\big( \sum_{j=1}^s \dim(W_j) \big)^2 > \dim(V)$.
	\end{itemize}
Each $W_j$ is also a composition factor of multiplicity one for the corresponding simple ideal $\fsl(W_j)$ in $\fh_V$, by Lemma \ref{lemma:simples-remain-simple} and Proposition \ref{prop:iso-equivalent}. Then $\fs_V$ is a simple Lie algebra by Lemma \ref{lemma:simple-ideal-criterion}. Together with Lemma \ref{lemma:V34} below (which is independent of the rest of this proof), this shows for all $V \in X$ that $\fs_V$ is a simple Lie algebra.

The next step of the proof is to show that
	\begin{equation} \label{eq:sV-H4}
	\fs_V = \begin{cases}
	\fsl(V) & \text{if $V \in E' \cup \set{V_{\refl},{}^\sigma V_{\refl}}$,} \\
	\osp(V) & \text{if $V \in F''$.}
	\end{cases}
	\end{equation}
Once again, assume for the moment that $V \in X$ and $V \neq V_{34}$. From the explicit determination of $\Res_H^G(V)$, one can calculate the rank of the semi\-simple Lie algebra $\fh_V$:
	\[
	\rk(\fh_V) = \sum_{\substack{W \in \Irr^{>1}(H)/\sim \\ W \leq V \text{ or } W^{\ve} \leq V}} (\dim(W)-1).
	\]
Since $\fh_V \subseteq \fs_V$, we have $\rk(\fh_V) \leq \rk(\fs_V)$. By case-by-case inspection, one observes (for $V \neq V_{34}$) that $\dim(V) < 4 \cdot \rk(\fh_V) \leq 4 \cdot \rk(\fs_V)$. There are now several cases to determine $\fs_V$:
	\begin{enumerate}
	\item \label{item:E'-dim-9-16} If $\dim(V) < 2 \cdot \rk(\fs_V)$, then $\fs_V = \fsl(V)$ by \cite[Lemme 13]{Marin:2007}. This condition is seen to hold if $V \in E'/\!\!\sim$ and $\dim(V) \in \set{9, 16}$, since then $\dim(V) < 2 \cdot \rk(\fh_V) \leq 2 \cdot \rk(\fs_V)$.
	\item \label{item:E'-dim-4-25-36} If $V \in E'/\!\!\sim$ and $\dim(V) \in \set{4,25,36}$ (the dimensions not covered by the previous case), then $2 \cdot \rk(\fh_V) \leq \dim(V)$. Suppose for the sake of a contradiction that $2 \cdot \rk(\fs_V) \leq \dim(V)$. Then $2 \cdot \rk(\fs_V) \leq \dim(V) < 4 \cdot \rk(\fs_V)$, so we may apply \cite[Lemme 14]{Marin:2007} to deduce that either $(\fs_V,V) \cong (\fsp(4),\C^4)$, or $(\fs_V,V) \cong (\fso(25),\C^{25})$, or $(\fs_V,V) \cong (\fsp(36),\C^{36})$.\footnote{Note that the ``additional'' exceptional case $(B_2,\omega_2)$ listed in \cite[Lemme 14]{Marin:2007} coincides with the case $(C_2,\omega_1)$ via the identification of the simple Lie algebras of types $B_2$ and $C_2$.} This implies that there exists a non-degenerate $\fs_V$-invariant bilinear form on $V$ (symmetric in the second case, and anti-symmetric in the first and third cases). Then $V \cong V^\ve$ by Lemma \ref{lemma:form-implies-dual}, contradicting the assumption that $V \in E'/\!\!\sim$. Thus we must have $\dim(V) < 2 \cdot \rk(\fs_V)$, and hence $\fs_V = \fsl(V)$.

	\item If $V \in F''$, then $\fs_V \subseteq \osp(V)$ by Lemma \ref{lemma:osp-factorization}, so $\rk(\fs_V) \leq \frac{1}{2} \cdot \dim(V)$, and hence (for $V \neq V_{34}$) one gets
		\[
		2 \cdot \rk(\fs_V) \leq \dim(V) < 4 \cdot \rk(\fs_V).
		\]
	Again applying the classification stated in \cite[Lemme 14]{Marin:2007}, we deduce that $\fs_V = \osp(V)$, except perhaps if $\dim(V) \in \set{8,10,16}$.
		\begin{itemize}
		\item If $\dim(V) = 8$, it may be that $\fs_V \cong \fso(7)$, or that $\fs_V \cong \fso(8)$. By Lemma \ref{lemma:H4-V9-V10-factorization} above, we have that $\dim(\fs_V) = 28$, implying that $\fs_V = \fso(V) \cong \fso(8)$.

		\item If $\dim(V) = 10$, it may be that $\fs_V \cong \fsl(5)$, or that $\fs_V \cong \fso(10)$. Using GAP, we have calculated that $\dim(\fs_V) = 45$, implying that $\fs_V = \fso(V) \cong \fso(10)$.

		\item If $\dim(V) = 16$, it may be that $\fs_V \cong \fso(10)$, or that $\fs_V \cong \fso(16)$. Using GAP, we have calculated that $\dim(\fs_V) = 120$, implying that $\fs_V = \fso(V) \cong \fso(16)$.
		\end{itemize}
	\end{enumerate}
This completes the verification of \eqref{eq:sV-H4} for all $V \in X$ except for $V = V_{34}$. For $V \in F''$ with $\dim(V) \in \set{8,10,16}$, we have further determined that $\fs_V = \fso(V)$.

Let $X' = \set{ V \in X : V \neq V_{34}}$. Then $I := \sum_{V \in X'} \fs^V$ is a direct sum in $\fD(\fs)$, by Lemma \ref{lemma:distinct-simple-ideals}, and we may (uniquely) write $\fD(\fs) = I \oplus J$, where $J$ is a (possibly empty) sum of simple ideals in $\fD(\fs)$. Using the dimensions of the simple $\CG$-modules given in Table \ref{tab:H4}, and noting the two possible values for $\dim(\osp(V))$ in the cases where the type of $\osp(V)$ has not yet been determined, it follows that $5839 \leq \dim(I) \leq 6053$. On the other hand, using GAP we have calculated\footnote{Our calculation took around 92 hours to complete.} that $\dim(\fs) = 6968$, and hence $\dim(\fD(\fs)) = 6967$ by Lemma \ref{lemma:g-reductive}. Then $J \neq 0$, and $914 \leq \dim(J) \leq 1128$.

From the uniqueness of the decomposition of a semisimple Lie algebra into simple ideals, it follows that $J \subseteq \bigcap_{V \in X'} \ker(\rho_V)$. Thus the image of $J$ under the map \eqref{eq:H4-D(s)-factorization} is contained entirely in the factor $\osp(V_{34})$. This image must be nonzero because $J \neq 0$ and \eqref{eq:H4-D(s)-factorization} is an injection. Then $J \not\subseteq \ker(\rho_{V_{34}})$, so $J \cap \fs^{V_{34}} \neq 0$, and hence $\fs^{V_{34}} \subseteq J$ by the simplicity of $\fs^{V_{34}}$. Then $J$ can be (uniquely) written as $J = \fs^{V_{34}} \oplus J'$ for some (possibly empty) sum of simple ideals $J'$. It follows that $J' \subseteq \bigcap_{V \in X} \ker(\rho_V)$, implying that $J' = 0$ because \eqref{eq:H4-D(s)-factorization} is an injection.

Now $J = \fs^{V_{34}} \cong \fs_{V_{34}}$ is a simple Lie algebra of dimension at least $914$. By inspection, one sees that this is possible only if $\rk(\fs_{V_{34}}) \geq 22$, and hence $4 \cdot \rk(\fs_{V_{34}}) > 48 = \dim(V_{34})$. Since $V_{34} \in F''$, we also have $2 \cdot \rk(\fs_{V_{34}}) \leq \dim(V_{34})$, and thus we may apply the classification stated in \cite[Lemme 14]{Marin:2007} to deduce that $\fs_{V_{34}} = \osp(V_{34})$.

We have now shown that $\fD(\fs) = \bigoplus_{V \in X} \fs^V$, and have shown for each $V \in X$ that $\fs^V \cong \fs_V$ is of the form stated in \eqref{eq:sV-H4}. This implies that the map \eqref{eq:H4-D(s)-factorization} is a Lie algebra isomorphism. Further, since $\dim(\fD(\fs)) = 6967$ as remarked above, it follows by dimension comparison that one must have $\fs_V = \fso(V)$ for each $V \in F''$.
\end{proof}

\begin{lemma} \label{lemma:V34}
Let $G = W(H_4)$, and let $V$ be the $48$-dimensional simple $\CG$-module, labeled $V_{34}$ in Table \ref{tab:H4}. Then $\fs_V$ is a simple Lie algebra of rank at least $11$.
\end{lemma}

\begin{proof}
We retain the notational conventions of Theorem \ref{thm:H4-D(s)-LA}. Then in the notation of Table \ref{tab:H3},
	\[
	\Res^G_H(V) = (U_3 \oplus U_5) \oplus (U_4 \oplus U_6) \oplus (U_7 \oplus U_8)^{\oplus 2} \oplus (U_9 \oplus U_{10})^{\oplus 2},
	\]
and hence
	\[
	\fh_V \cong \fsl(U_3) \oplus \fsl(U_4) \oplus \fsl(U_7) \oplus \fsl(U_9) \cong \fsl(3) \oplus \fsl(3) \oplus \fsl(4) \oplus \fsl(5).
	\]
The simple ideal $\fh_1 := \fsl(U_3)$ has the multiplicity-one composition factors $U_3$ and $U_5$ in $V$, and $\fh_2 := \fsl(U_4)$ has the multiplicity-one composition factors $U_4$ and $U_6$. Each nontrivial simple factor for the simple ideals $\fh_3 := \fsl(U_7)$ and $\fh_4 := \fsl(U_9)$ occurs with multiplicity two in $V$.

First we show that the Lie algebra $\fs_V$ is simple. Let $\fs_V = \g_1 \oplus \cdots \oplus \g_t$ be the decomposition of $\fs_V$ into simple ideals as in Lemma \ref{lemma:simple-ideal-criterion}, and let $V = V_1 \boxtimes \cdots \boxtimes V_t$ be the corresponding expression of $V$ as an external tensor product of simple modules. Considering the simple ideals $\fh_1$ and $\fh_2$ and arguing as in the proof of Lemma \ref{lemma:simple-ideal-criterion}, one sees for each $1 \leq i \leq t$ that $\dim(V_i) \geq \dim(U_3) + \dim(U_4) = 6$. Since $48 = \dim(V) = \prod_{i=1}^t \dim(V_i)$, this implies that $t \leq 2$. Now suppose for the sake of a contradiction that $t = 2$. Let $i \in \set{1,2}$. If $V_i$ were trivial as an $\fh_3$-module, it would follow as in the proof of Lemma \ref{lemma:simple-ideal-criterion} that each $\fh_3$-module composition factor in $V$ is of multiplicity at least $\dim(V_i) \geq 6$, contradicting the observation that each nontrivial $\fh_3$-module factor in $V$ is of multiplicity two. Then $V_i$ is nontrivial as an $\fh_3$-module. The same conclusion holds for $V_i$ as an $\fh_4$-module. Then considering $V_i$ as a module for the semisimple Lie algebra $\fh' = \fh_1 \oplus \fh_2 \oplus \fh_3 \oplus \fh_4$, and arguing as in the proof of the lemma, one sees that $\dim(V_i) \geq 3 + 3 + 4 + 5 = 15$. Then
	\[
	48 = \dim(V) = \dim(V_1) \cdot \dim(V_2) \geq 15 \cdot 15 = 225,
	\]
a contradiction. Thus $t = 1$, and $\fs_V$ is a simple Lie algebra with $\rk(\fs_V) \geq \rk(\fh_V) = 11$.
\end{proof}

\begin{remark} \label{remark:triality}
Let $G = W(H_4)$, and let the simple $\CG$-modules be labeled as in Table \ref{tab:H4}. The preceding proof together with Lemma \ref{lemma:ext-powers-sigma-twist} implies that $\fs^{V_9} = \fs^{V_{10}} \cong \fso(8)$. The Lie algebra $\fso(8)$ has three $8$-dimensional irreducible representations: its defining representation and its two spin representations. Fixing an identification of $\fso(8)$ with $\fs^{V_9} = \fs^{V_{10}}$, we do not know which two out of those three modules correspond to $V_9$ and $V_{10}$.
\end{remark}

\subsection{Type \texorpdfstring{$F_4$}{F4}}

Let $G = W(F_4)$ be the reflection group of type $F_4$. In this case the injective Lie algebra map \eqref{eq:D(s)-factorization} admits a further proper factorization. First we recall from \cite{Kondo:1965} some details on the structure and representation theory of the group $G$. Our notation is based on that in \cite{Kondo:1965}.

\begin{figure}[htbp]
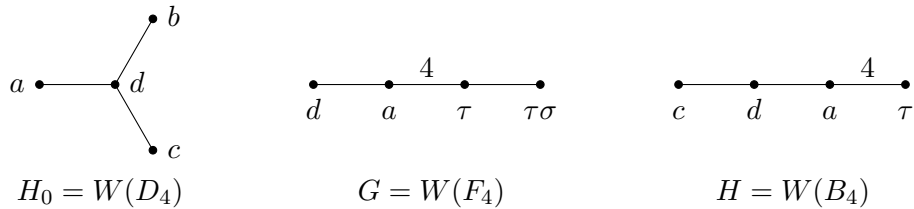

\setlength{\tabcolsep}{18pt}
\begin{tabular}{cccc}
$\dynkin[Coxeter,text style/.style={scale=1},edge length=1cm,labels={a,d,b,c},label directions={left,right, , }]D4$ & 
$\dynkin[Coxeter,text style/.style={scale=1},edge length=1cm,labels={d,a,\tau,\tau\sigma},label directions={,,,}]F4$ & 
$\dynkin[Coxeter,text style/.style={scale=1},edge length=1cm,labels={c,d,a,\tau},label directions={,,,}]B4$ \\
$H_0 = W(D_4)$ & $G = W(F_4)$ & $H = W(B_4)$ \\
\end{tabular}
\caption{Coxeter diagrams and generators for the groups $H_0$, $G$, and $H$.} \label{fig:Coxeter-diagrams}
\end{figure}

Coxeter generators for the groups $H_0 = W(D_4)$, $G = W(F_4)$, and $H = W(B_4)$ are specified via the labelings of the Coxeter diagrams in Figure \ref{fig:Coxeter-diagrams}. Let $\sigma = \tau \cdot \tau \sigma$, and let $S_2 = \subgrp{\tau,\sigma} \cong \fS_3$. One has $G = H_0 \rtimes S_2$, with the action of $S_2$ on $H_0$ defined by
	\begin{equation} \label{eq:S3-on-D4-relations}
	\begin{gathered}
	\tau a \tau^{-1} = b, \quad \tau c = c \tau, \quad \tau d = d \tau, \\
	\sigma a \sigma^{-1} = b, \quad \sigma b \sigma^{-1} = c, \quad \sigma c \sigma^{-1} = a, \quad \sigma d \sigma^{-1} = d.
	\end{gathered}
	\end{equation}
In this manner, the groups $H_0$ and $H$ are each subgroups of $G$, and $H_0 \leq H$. The group $G$ admits four linear characters: the trivial character, and characters $\ve$, $\ve'$, and $\ve''$ defined by
	\begin{equation} \label{eq:F4-linear-characters}
	\begin{aligned}
	\ve(d) &= -1, \qquad & \ve(a) &= -1, \qquad & \ve(\tau) &= -1, \qquad & \ve(\tau\sigma) &=-1, \\
	\ve'(d) &= +1, \qquad & \ve'(a) &= +1, \qquad & \ve'(\tau) &= -1, \qquad & \ve'(\tau\sigma) &=-1, \\
	\ve''(d) &= -1, \qquad & \ve''(a) &= -1, \qquad & \ve''(\tau) &= +1, \qquad & \ve''(\tau\sigma) &=+1.
	\end{aligned}
	\end{equation}
The notation for the linear characters is chosen so that their restrictions to $H$ and $H_0$ are consistent with the notation in \cite{DK:2026a}. Note that $\ve$ is the sign character of $G$, and $\ve = \ve' \otimes \ve''$.
	
There exists an involutive automorphism $\iota$ of $G$, corresponding to the symmetry of the Coxeter diagram of $G$, such that $\iota(d) = \tau \sigma$, $\iota(a) = \tau$, $\iota(\tau) = a$, and $\iota(\tau \sigma) = d$. Let $N = \iota(H_0) \cap H_0$. Then $N \unlhd G$, $N \subset [G,G]$, and the diagonal map $G \to G \times G$ induces an isomorphism
	\[
	G/N \cong G/\iota(H_0) \times G/H_0 \cong \fS_3 \times \fS_3.
	\]
Let $e = (abcd)^2$, an element of order $3$ in $G$, and let $S_1 = \subgrp{d,e}$. Then $S_1 \cong S_2 \cong \fS_3$, the subgroup generated by $S_1$ and $S_2$ is a direct product in $G$, and $G = N \rtimes (S_1 \times S_2)$. Note that $\ve''$ and $\ve'$ restrict to the sign characters on $S_1$ and $S_2$, respectively.

From the relations \eqref{eq:S3-on-D4-relations}, one sees that $e$ is equivalent to $da$ modulo the normal subgroup $\iota(H_0)$, and hence $G = \iota(H_0) \rtimes S_1$. Let $\pi_1 : G \twoheadrightarrow S_1 \cong \fS_3$ and $\pi_2 : G \twoheadrightarrow S_2 \cong \fS_3$ be the quotient maps with kernels $\iota(H_0) = \subgrp{\tau,\iota(b),\iota(c),\tau\sigma}$ and $H_0$, respectively. Given a $\CS_3$-module $V$, write $\pi_i^*(V)$ for the $\CG$-module obtained via inflation along $\pi_i$. Denote the simple $\CS_3$-modules by
	\begin{itemize}
	\item $S^{[3]}$, the trivial module;
	\item $S^{[2,1]}$, the two-dimensional reflection representation of $\fS_3$; and
	\item $S^{[1^3]}$, the sign representation of $\fS_3$.
	\end{itemize}
Note that $S^{[2,1]} \otimes S^{[1^3]} \cong S^{[2,1]}$, and if $V$ and $W$ are $\CS_3$-modules, then
	\begin{equation} \label{eq:inflation-tensor-sign}
	\begin{split}
	[\pi_1^*(V) \otimes \pi_2^*(W)] \otimes \ve &= [\pi_1^*(V) \otimes \ve''] \otimes [\pi_2^*(W) \otimes \ve'] \\
	&= \pi_1^*(V \otimes S^{[1^3]}) \otimes \pi_2^*(W \otimes S^{[1^3]}).
	\end{split}
	\end{equation}

\begin{lemma} \label{lemma:D(s)-F4-injection}
Label the simple $\CG$-modules as in Table \ref{tab:F4}. Set $\fs_i = \fs_{V_i}$ and $\rho_i = \rho_{V_i}$.
	\begin{enumerate}
	\item \label{item:refl-powers-translate} Let $\kappa \in \set{\ve',\ve''}$. Then $\set{V_9 \otimes \kappa, V_{14} \otimes \kappa, V_{12} \otimes \kappa} = \set{V_{10},V_{15},V_{11}}$.

	\item \label{item:rho13-factors} One has $V_{13} = V_5 \otimes V_7$, and via this identification, there exists a commutative diagram of Lie algebras
		\[
		\vcenter{\xymatrix{
		\fD(\fs) \ar@{->}[r]^-{\rho_{5} \oplus \rho_{7}} \ar@{->}[d]_{\rho_{13}} & \fs_{5} \oplus \fs_{7} \ar@{->}[d]^{\rho_5 \boxtimes \rho_7} \\
		\fs_{13} \ar@{^(->}[r]^-{\subseteq} & \End(V_5 \otimes V_7),
		}}
		\]
	where $\rho_5 \boxtimes \rho_7$ is the structure map for $V_5 \otimes V_7$ as an external tensor product of modules for the Lie algebra $\fs_5 \oplus \fs_7$, defined by $(\rho_5 \boxtimes \rho_7)(x,y) = \rho_5(x) \otimes \id_{V_7} + \id_{V_5} \otimes \rho_7(y)$.

	\item \label{item:rho15-factors} The Lie algebra map $\rho_{15}: \fD(\fs) \to \fs_{15}$ factors through the map $\rho_{10}: \fD(\fs) \to \fs_{10}$.
	
	\item \label{item:D(s)-factorization-F4} Let $E$ and $F$ be the subsets of $\Irr(G)$ defined by \eqref{eq:E-F-Irr}, and let $F'$ be the subset of $F$ defined by omitting the irreducibles in $F$ of dimensions $4$ and $6$ (i.e., omitting $V_{13}$ and $V_{15}$). Then \eqref{eq:D(s)-factorization} restricts further to an injective Lie algebra homomorphism
	\begin{equation} \label{eq:D(s)-F4-injection}
	\fD(\fs) \hookrightarrow \fsl(V_{\refl}) \oplus \Big[ \bigoplus_{V \in E/\sim} \fsl(V) \Big] \oplus \Big[ \bigoplus_{V \in F'} \osp(V) \Big].
	\end{equation}
	\end{enumerate}
\end{lemma}

\begin{proof}
Statement \eqref{item:refl-powers-translate} is checked by character calculations. To establish \eqref{item:rho13-factors}, first recall from \cite[\S3.1]{Kondo:1965} that $G$ has nine simple modules that are trivial for $N$: the four $1$-dimensional modules, the four $2$-dimensional modules
	\begin{align*}
	V_5 &= \pi_1^*(S^{[2,1]}) = \pi_1^*(S^{[2,1]}) \otimes \pi_2^*(S^{[3]}), & V_6 &= \pi_1^*(S^{[2,1]}) \otimes \pi_2^*(S^{[1^3]}), \\
	V_7 &= \pi_2^*(S^{[2,1]}) = \pi_1^*(S^{[3]}) \otimes \pi_2^*(S^{[2,1]}), & V_8 &= \pi_1^*(S^{[1^3]}) \otimes \pi_2^*(S^{[2,1]}),
	\end{align*}
and the module $V_{13} = V_5 \otimes V_7 = \pi_1^*(S^{[2,1]}) \otimes \pi_2^*(S^{[2,1]})$; the identification $V_{13} = V_5 \otimes V_7$ follows from \eqref{eq:inflation-tensor-sign} and the fact that $G$ has a unique $4$-dimen\-sional simple module $V$ such that $V \cong V^\ve$.

Consider the structure maps for the group-theoretic and Lie-theoretic actions of $\fs$ on the tensor product $V_5 \otimes V_7$, i.e., $\rho_{\Grp} = \rho_{5} \otimes \rho_{7} = \rho_{13}$ and $\rho_{\Lie} = \rho_5 \otimes \id_{V_7} + \id_{V_5} \otimes \rho_7$. By Remark \ref{remark:s-generated-by-S0}, the Lie algebra $\fs$ is generated by the set $S_0 = \set{d,a,\tau,\tau \sigma}$. For each $s \in S_0$, one sees that
	\[
	\rho_{\Lie}(s) = \rho_{\Grp}(s) + \id_{V_5 \otimes V_7}.
	\]
For example, since $d \in H_0$ and since $H_0$ acts trivially on $V_7$, one has
	\begin{gather*}
	\rho_{\Lie}(d) = \rho_5(d) \otimes \id_{V_7} + \id_{V_5} \otimes \rho_7(d) = \rho_5(d) \otimes \id_{V_7} + \id_{V_5 \otimes V_7}, \text{ and} \\
	\rho_{\Grp}(d) = \rho_5(d) \otimes \rho_7(d) = \rho_5(d) \otimes \id_{V_7}.
	\end{gather*}
Since the values of $\rho_{\Grp}$ and $\rho_{\Lie}$ on a set of generators for $\fs$ differ by a central operator, it follows that the restrictions of $\rho_{\Grp}$ and $\rho_{\Lie}$ to $\fD(\fs)$ are equal. This establishes \eqref{item:rho13-factors}.

To establish \eqref{item:rho15-factors}, we first recall some details from \cite[\S 3.3]{Kondo:1965}. The group $H = W(B_4)$ has two simple modules, $U_1$ and $U_2 = U_1 \otimes \ve'$, whose restrictions to $H_0 = W(D_4)$ are simple; these are the modules denoted $S^{([2],[1^2])}$ and $S^{([1^2],[2])}$ in \cite{DK:2026a}. The induced modules $U_1^* := \Ind_H^G(U_1)$ and $U_2^* := \Ind_H^G(U_2)$ decompose as $U_1^* = U_1' \oplus V_{24}$ and $U_2^* = U_2' \oplus V_{24}$, where $V_{24}$ is the simple $\CG$-module of dimension $12$, and $U_1'$ and $U_2'$ are the simple $\CG$-modules of dimension $6$. Reordering if necessary, we may assume that $U_1' = V_{14}$ and $U_2' = V_{15}$. Note that, since the character $\ve'$ lifts to $G$, and since $G$ has a unique simple module of dimension $12$, one has
	\[
	U_2^* = \Ind_H^G(U_1 \otimes \ve') \cong \Ind_H^G(U_1) \otimes \ve' = (U_1' \otimes \ve') \oplus (V_{24} \otimes \ve') \cong (U_1' \otimes \ve') \oplus V_{24},
	\]
and hence $U_2' \cong U_1' \otimes \ve'$.

Next, recall from \cite[\S 3.4]{Kondo:1965} that, after $V_{13}$, the remaining simple $\CG$-modules of dimension $4$ can be obtained from one another by tensoring with a linear character of $G$. To wit, one has the reflection representation $V_9$ and the modules $V_{12} = V_9 \otimes \ve$, $V_{10} = V_9 \otimes \ve'$, and $V_{11} = V_9 \otimes \ve''$. Since $U_1' = V_{14} = \Lambda^2(V_9)$, we have $V_{15} = U_2' \cong U_1' \otimes \ve' = \Lambda^2(V_9) \otimes \ve'$. We claim that $V_{15} \cong \Lambda^2(V_{10})$ as $\fD(\fs)$-modules. More precisely, we claim that via the canonical vector space identifications $\Lambda^2(V_9) \otimes \ve' = \Lambda^2(V_9) = \Lambda^2(V_9 \otimes \ve')$, one has
	\[
	\Res^{\CG}_{\fD(\fs)}( \Lambda^2(V_9) \otimes \ve') = \Lambda^2( \Res^{\CG}_{\fD(\fs)}(V_9 \otimes \ve')).
	\]
This is established by an argument parallel to that given for Lemma \ref{lemma:ext-powers-refl-rep}, the only difference being that for the ``edge'' simple reflection $\tau \sigma$, one obtains the relation $\rho_{\Lie}^{\wedge^2}(\tau\sigma) = \rho_{\Grp}^{\wedge^2}(\tau\sigma) - \id_{\Lambda^2(V_9)}$ rather than the relation $\rho_{\Lie}^{\wedge^2}(\tau\sigma) = \rho_{\Grp}^{\wedge^2}(\tau\sigma) + \id_{\Lambda^2(V_9)}$; cf.\ the proof of \cite[Lemma 4.3.7]{DK:2026a}. This still implies for all simple reflections $s$ that $\rho_{\Lie}^{\wedge^2}(s)$ and $\rho_{\Grp}^{\wedge^2}(s)$ differ by a central element of $\End(\Lambda^2(V_9))$, and hence implies that the restrictions of $\rho_{\Grp}^{\wedge^2}$ and $\rho_{\Lie}^{\wedge^2}$ to $\fD(\fs)$ are the same. Then the factorization stated in part \eqref{item:rho15-factors} follows as in Corollary \ref{cor:exterior-powers-factorization}.

Finally, part \eqref{item:D(s)-factorization-F4} is an immediate consequence of parts \eqref{item:rho13-factors} and \eqref{item:rho15-factors}.
\end{proof}

\begin{lemma} \label{lemma:F4-sl(2)}
Let $G = W(F_4)$, and let $V \in \Irr(G)$ with $\dim(V)=2$. Then $\fs_V = \fsl(V)$.
\end{lemma}

\begin{proof}
By construction, $\fs_V$ is a semisimple subalgebra of the $3$-dimensional Lie algebra $\fsl(V)$. The Lie algebra $\fs_V$ must be non\-zero, because $V$ is a simple $\fD(\fs)$-module (by Lemma \ref{lemma:simples-remain-simple}) of dimension greater than one, and $\fs_V$ cannot be abelian, because $V$ is simple of dimension greater than one. The only nonzero, non-abelian, semisimple Lie algebra of dimension at most $3$ is $\fsl(2)$, so $\fs_V = \fsl(V)$.
\end{proof}

\begin{theorem} \label{thm:F4-D(s)-LA}
Let $G = W(F_4)$. Label the simple $\CG$-modules as in Table \ref{tab:F4}, and let $F'$ be the subset of $F$ defined by omitting the simple modules of dimensions $4$ and $6$ (i.e., $V_{13}$ and $V_{15}$). Then the map \eqref{eq:D(s)-F4-injection} defines an isomorphism
	\begin{gather}
	\fD(\fs) \cong \fsl(V_{\refl}) \oplus \Big[ \bigoplus_{V \in E/\sim} \fsl(V) \Big] \oplus \Big[ \bigoplus_{V \in F'} \fso(V) \Big] \label{eq:D(s)-F4} \\
	\cong \fsl(2)^{\times 2} \oplus \fsl(4)^{\times 2} \oplus \fsl(8)^{\times 2} \oplus \fsl(9)^{\times 2} \oplus \fso(12) \oplus \fso(16). \notag
	\end{gather}
\end{theorem}

\begin{proof}
Let $X = \set{V_{\refl}} \cup (E/\!\!\sim) \cup F'$. The first step of the proof is to show for all $V \in X$ that $\fs_V$ is a simple Lie algebra. We have already shown in Lemma \ref{lemma:F4-sl(2)} that $\fs_V = \fsl(V)$ if $\dim(V) = 2$, so we may assume that $\dim(V) > 2$. Let $S_0 = \set{d,a,\tau,\tau\sigma}$, and let $\calB \cong W(B_3)$ be the subgroup of $G$ generated by the set $\calB_0 := \set{d,a,\tau}$.\footnote{The subgroup generated by $\set{a,\tau,\tau\sigma}$ provides a different copy of $W(B_3)$ in $G$; the choice is immaterial.} We will show that $\fs_V$ is simple by considering $\Res^G_\calB(V)$, the restriction of $V$ to the subgroup $\calB$.

Let $\fb \subset \fs$ be the Lie subalgebra of $\Lie(\CB) \subset \Lie(\CG)$ generated by the set $\calB_0$. Given $V \in \Irr(G)$ with module structure map $\rho_V : \CG \to \End(V)$, set $\fb_V = \rho_V(\fD(\fb))$. Then $\fb_V \subseteq \fs_V$ are semi\-simple subalgebras of $\fsl(V)$, and $\fs_V$ acts irreducibly on $V$ by Lemma \ref{lemma:simples-remain-simple}. For $W \in \Irr(\calB)$, set $\fb_W = \rho_W(\fD(\fb))$. It follows from the results of \cite{DK:2026a} that $\fb_W = \fsl(W)$ for each $W \in \Irr(\calB)$, and
	\[
	\fb_V \cong \bigoplus_{\substack{W \in \Irr^{>1}(\calB)/\sim \\ W \leq V \text{ or } W^{\ve} \leq V}} \fsl(W).
	\]
In other words, $\fb_V$ consists of a copy of $\fsl(W)$ for each $W \in \Irr^{>1}(\calB)/\!\!\sim$ such that either $W$ or $W^{\ve}$ occurs with multiplicity at least one in $\Res^G_\calB(V)$.

Fix $V \in X$ with $\dim(V) > 2$. The composition factor multiplicities of $\Res^G_\calB(V)$ can be explicitly determined, for example by character calculations using GAP \cite{GAP4}. By inspection, one observes that there exist $\CB$-modules $W_1,\ldots,W_s \in \Irr^{>1}(\calB )/\!\!\sim$ such that the following conditions hold:
	\begin{itemize}
	\item For each $1 \leq j \leq s$, either $W_j$ or $(W_j)^\ve$ occurs with multiplicity one in $\Res^G_\calB(V)$.
	\item $\big( \sum_{j=1}^s \dim(W_j) \big)^2 > \dim(V)$.
	\end{itemize}
Each $W_j$ is also a composition factor of multiplicity one for the corresponding simple ideal $\fsl(W_j)$ in $\fb_V$, by Lemma \ref{lemma:simples-remain-simple} and Proposition \ref{prop:iso-equivalent}. Then $\fs_V$ is simple by Lemma \ref{lemma:simple-ideal-criterion}.

Next, in the notation of Table \ref{tab:F4}, one has $F' = \set{V_{24}, V_{25}}$. By direct calculation in GAP, we have computed that $\dim(\fs_{V_{24}}) = 66$ and $\dim(\fs_{V_{25}}) = 120$. The only simple Lie algebra of dimension $66$ is the rank-$6$ algebra $\fso(12)$, so we deduce by Lemma \ref{lemma:osp-factorization} and \cite[Lemme 14]{Marin:2007} that $\fs_{V_{24}} = \fso(V_{24})$. Similarly, the only simple Lie algebras of dimension $120$ are $\fsl(11)$ of rank $10$ and $\fso(16)$ of rank $8$, so we deduce by Lemma \ref{lemma:osp-factorization} and \cite[Lemme 14]{Marin:2007} that $\fs_{V_{25}} = \fso(V_{25})$. Together, these results imply that \eqref{eq:D(s)-F4-injection} defines an injective Lie algebra map
	\begin{equation} \label{eq:D(s)-F4-injection-so}
	\fD(\fs) \hookrightarrow \fsl(V_{\refl}) \oplus \Big[ \bigoplus_{V \in E/\sim} \fsl(V) \Big] \oplus \Big[ \bigoplus_{V \in F'} \fso(V) \Big].
	\end{equation}
Using the dimensions in Table \ref{tab:F4}, we see that the right-hand side of \eqref{eq:D(s)-F4-injection-so} is of dimension $508$. On the other hand, using GAP \cite{GAP4}, we have calculated that $\dim(\fs) = 510$, and hence $\dim(\fD(\fs)) = 508$. Then \eqref{eq:D(s)-F4-injection-so} must be a surjection, and hence an isomorphism.
\end{proof}

\subsection{Type \texorpdfstring{$E_6$}{E6}}

Let $G = W(E_6)$ be the reflection group of type $E_6$.

\begin{theorem} \label{thm:D(E6)-LA}
For the reflection group of type $E_6$, the map \eqref{eq:D(s)-factorization} defines an isomorphism
	\begin{equation} \label{eq:D(E6)-LA}
	\fD(\fs) \cong \fsl(V_{\refl}) \oplus \Big[ \bigoplus_{V \in E/\sim} \fsl(V) \Big] \oplus \Big[ \bigoplus_{V \in F} \osp(V) \Big].
	\end{equation}
\end{theorem}

\begin{proof}
Let $X = \set{V_{\refl}} \cup (E/\!\!\sim) \cup F$. The first main step of the proof is to show for all $V \in X$ that $\fs_V$ is a simple Lie algebra. Let $S_0 = \set{s_1,\ldots,s_5,s_6}$ be a set of Coxeter generators for $G$ such that the corresponding Coxeter graph is labeled as in \cite[Theorem 1.3.3]{Geck:2000}, and let $H \cong W(D_5)$ be the subgroup of $G$ generated by the set $\set{s_1,\ldots,s_5}$. We will show that $\fs_V$ is simple by considering $\Res_H^G(V)$, the restriction of $V$ to the subgroup $H$.

Let $\fh \subset \fs$ be the Lie subalgebra of $\Lie(\CH) \subset \Lie(\CG)$ generated by the set of reflections in $H$. Given $V \in \Irr(G)$ with module structure map $\rho_V : \CG \to \End(V)$, set $\fh_V = \rho_V(\fD(\fh))$. Then $\fh_V \subseteq \fs_V$ are semisimple Lie subalgebras of $\fsl(V)$, and $\fs_V$ acts irreducibly on $V$, by Lemma \ref{lemma:simples-remain-simple}. For $U \in \Irr(H)$, set $\fh_U = \rho_U(\fD(\fh))$. Label the elements of $\Irr(H)$ as in Table \ref{tab:D5}, and let
	\begin{gather*}
	E_0 = \set{U_7,U_{13},U_{11},U_8} = \set{ U_7, \Lambda^2(U_7), \Lambda^3(U_7), \Lambda^4(U_7)}, \\
	E_1 = \set{ U_3, U_9, U_4} = \set{U_3, \Lambda^2(U_3), \Lambda^3(U_3)}, \\
	E_2 = \set{U_5,U_6,U_{10},U_{12},U_{15},U_{16},U_{17},U_{18}}.
	\end{gather*}
Note that $E(H) = E_2 \cup \set{U_3,U_4}$, and $F(H) = \set{U_9,U_{14}}$. For $i \in \set{0,1}$, write $E_i \cap V \neq \emptyset$ to mean that $\Res_H^G(V)$ contains at least one simple $\CH$-module factor from the set $E_i$. Then it follows from the results of \cite{DK:2026a} that
	\begin{equation} \label{eq:hV-E6-to-D5}
	\fh_V 
	\cong \Big[ \bigoplus_{E_0 \cap V \neq \emptyset} \fsl(U_7) \Big] 
	\oplus \Big[ \bigoplus_{E_1 \cap V \neq \emptyset} \fsl(U_3) \Big] 
	\oplus \Big[ \bigoplus_{\substack{U \in E_2/\sim \\ U \leq V \text{ or } U^\ve \leq V}} \fsl(U) \Big] 
	\oplus \Big[ \bigoplus_{U_{14} \leq V} \fsp(U_{14}) \Big].
	\end{equation}
In other words, $\fh_V$ is the direct sum of:
	\begin{itemize}
	\item a single copy of $\fsl(U_7)$, if $E_0 \cap V \neq \emptyset$;
	\item a single copy of $\fsl(U_3)$, if $E_1 \cap V \neq \emptyset$;
	\item a single copy of $\fsl(U)$, for each $U \in E_2/\!\!\sim$ such that either $U \leq V$ or $U^{\ve} \leq V$;
	\item a single copy of $\fsp(U_{14})$, if $U_{14} \leq V$.
	\end{itemize}
The composition factor multiplicities of $\Res_H^G(V)$ can be explicitly determined, for example by character calculations in GAP \cite{GAP4}.

Now let $V \in X$ be fixed. First suppose $V$ is the unique $10$-dimensional simple $\CG$-module. In this case one has $\Res_H^G(V) = U_{14}$. Since $\fsp(U_{14}) = \fh_{U_{14}} = \fh_V \subseteq \fs_V \subseteq \osp(V)$, this implies that $\fs_V = \fsp(V)$. In particular, $\fs_V$ is simple. Now suppose $\dim(V) \neq 10$. Then by inspection, one observes that there exist simple $\CH$-modules $W_1,\ldots,W_s \in E_2 \cup \set{U_7}$ such that:
	\begin{itemize}
	\item For each $1 \leq j \leq s$, either $W_j$ or $(W_j)^\ve$ occurs with multiplicity one\footnote{In fact, one sees that restriction from $G = W(E_6)$ to $H = W(D_5)$ is multiplicity free.} in $\Res^G_H(V)$.
	\item $\big( \sum_{j=1}^s \dim(W_j) \big)^2 > \dim(V)$.
	\end{itemize}
Each $W_j$ is also a composition factor of multiplicity one for the corresponding simple ideal $\fsl(W_j)$ in $\fh_V$, by Lemma \ref{lemma:simples-remain-simple} and Proposition \ref{prop:iso-equivalent}. Then $\fs_V$ is simple by Lemma \ref{lemma:simple-ideal-criterion}.

The next step of the proof is to show that
	\begin{equation} \label{eq:sV-E6}
	\fs_V = \begin{cases}
	\fsl(V) & \text{if $V \in E \cup \set{V_{\refl}}$,} \\
	\osp(V) & \text{if $V \in F$.}
	\end{cases}
	\end{equation}
We have already observed that $\fs_V = \fsp(V)$ if $\dim(V) = 10$, so we may assume that $\dim(V) \neq 10$. From \eqref{eq:hV-E6-to-D5} and the explicit determination of $\Res^G_H(V)$, one can calculate the rank of the semisimple Lie algebra $\fh_V$. For $V \in E \cup \set{V_{\refl}}$, one observes that $2 \cdot \rk(\fs_V) \geq 2 \cdot \rk(\fh_V) > \dim(V)$, and hence $\fs_V = \fsl(V)$ by \cite[Lemme 13]{Marin:2007}. For $V \in F$, one observes that $4 \cdot \rk(\fs_V) \geq 4 \cdot \rk(\fh_V) > \dim(V)$. Since also $\fs_V \subseteq \osp(V)$ and hence $\dim(V) \geq 2 \cdot \rk(\fs_V)$ by Lemma \ref{lemma:osp-factorization}, this implies by \cite[Lemme 14]{Marin:2007} that $\fs_V = \osp(V)$.

Now let $I = \sum_{V \in X} \fs^V$. This sum is a direct sum in $\fD(\fs)$, by Lemma \ref{lemma:distinct-simple-ideals}, and we can (uniquely) write $\fD(\fs) = I \oplus J$, where $J$ is a (possibly empty) sum of simple ideals in $\fD(\fs)$. By construction, $J \subseteq \bigcap_{V \in X} \ker(\rho_V)$. Then the image of $J$ under the (injective) map \eqref{eq:D(s)-factorization} is zero, implying that $J = 0$. Thus $\fD(\fs) = I = \bigoplus_{V \in X} \fs^V$, which implies the isomorphism \eqref{eq:D(E6)-LA}.
\end{proof}

\begin{remark} \label{remark:E6-osp-dont-know}
For the simple modules of type $F$ in \eqref{eq:D(E6)-LA}, we do not know in general whether $\osp(V)$ is an orthogonal Lie algebra or a symplectic Lie algebra.
\end{remark}

\subsection{Type \texorpdfstring{$E_7$}{E7}}

Let $G = W(E_7)$ be the reflection group of type $E_7$. Note in this case that the set $F(G)$ defined by \eqref{eq:E-F-Irr} is empty.

\begin{theorem} \label{thm:D(E7)-LA}
For the reflection group of type $E_7$, the map \eqref{eq:D(s)-factorization} defines an isomorphism
	\begin{equation} \label{eq:D(E7)-LA}
	\fD(\fs) \cong \fsl(V_{\refl}) \oplus \Big[ \bigoplus_{V \in E/\sim} \fsl(V) \Big].
	\end{equation}
\end{theorem}

\begin{proof}
The proof is parallel to the proof of Theorem~\ref{thm:D(E6)-LA} but more straightforward. Once again, let $X = \set{V_{\refl}} \cup (E/\!\!\sim)$. The first step is to show for all $V \in X$ that $\fs_V$ is a simple Lie algebra. Let $S_0 = \set{s_1,\ldots,s_6,s_7}$ be a set of Coxeter generators for $G$ such that the corresponding Coxeter graph is labeled as in \cite[Theorem 1.3.3]{Geck:2000}, and let $H \cong W(E_6)$ be the subgroup of $G$ generated by the set $\set{s_1,\ldots,s_6}$. We show that $\fs_V$ is simple by considering $\Res_H^G(V)$.

Let $\fh \subset \fs$ be the Lie subalgebra of $\Lie(\CH) \subset \Lie(\CG)$ generated by the set of reflections in $H$. Given $V \in \Irr(G)$ with module structure map $\rho_V : \CG \to \End(V)$, set $\fh_V = \rho_V(\fD(\fh))$. Then $\fh_V \subseteq \fs_V$ are semisimple Lie subalgebras of $\fsl(V)$, and $\fs_V$ acts irreducibly on $V$, by Lemma \ref{lemma:simples-remain-simple}. For $U \in \Irr(H)$, set $\fh_U = \rho_U(\fD(\fh))$. Let $U_{\refl}$ be the reflection representation of $H$, and let 
	\[
	E_0(H) = \{ \Lambda^d(U_{\refl}) : 0 < d < 6 \}.
	\]
Then it follows from Theorem \ref{thm:D(E6)-LA} that
	\begin{equation} \label{eq:hV-E7-to-E6}
	\fh_V 
	\cong \Big[ \bigoplus_{E_0(H) \cap V \neq \emptyset} \fsl(U_{\refl}) \Big] 
	\oplus \Big[ \bigoplus_{\substack{U \in E(H)/\sim \\ U \leq V \text{ or } U^\ve \leq V}} \fsl(U) \Big] 
	\oplus \Big[ \bigoplus_{\substack{U \in F(H) \\ U \leq V}} \osp(U) \Big].
	\end{equation}
The composition factor multiplicities of $\Res_H^G(V)$ can be explicitly determined, and the simple modules in the set $E_0(H)$ can be identified, by character calculations in GAP \cite{GAP4}.


Now let $V \in X$ be fixed. The simplicity of $\fs_V$ is verified by applying Lemma \ref{lemma:simple-ideal-criterion} as in the proof of Theorem~\ref{thm:D(E6)-LA}, replacing $E_2 \cup \set{U_7}$ there with $E(H) \cup \set{U_{\refl}}$. Next, from \eqref{eq:hV-E7-to-E6} and the explicit determination of $\Res^G_H(V)$, one can calculate the rank of the semisimple Lie algebra $\fh_V$. With the exception of the two $512$-dimensional simple $\CG$-modules, one sees for all $V \in E \cup \set{V_{\refl}}$ that $2 \cdot \rk(\fs_V) \geq 2 \cdot \rk(\fh_V) > \dim(V)$, and hence $\fs_V = \fsl(V)$ by \cite[Lemme 13]{Marin:2007}. If $V$ is one of the two $512$-dimensional simple $\CG$-modules, one sees that $\rk(\fh_V) = 254$, and hence $\dim(V) < 4 \cdot \rk(\fs_V)$. If it were the case that $2 \cdot \rk(\fs_V) \leq \dim(V)$, then one could argue as in case \eqref{item:E'-dim-4-25-36} in the proof of Theorem \ref{thm:H4-D(s)-LA} to arrive at a contradiction. Thus it must be the case that $2 \cdot \rk(\fs_V) > \dim(V)$, and hence $\fs_V = \fsl(V)$.

Now let $I = \sum_{V \in X} \fs^V$. This sum is a direct sum in $\fD(\fs)$, by Lemma \ref{lemma:distinct-simple-ideals}, and we may (uniquely) write $\fD(\fs) = I \oplus J$, where $J$ is a (possibly empty) sum of simple ideals in $\fD(\fs)$. By construction, $J \subseteq \bigcap_{V \in X} \ker(\rho_V)$. Then the image of $J$ under the (injective) map \eqref{eq:D(s)-factorization} is zero, implying that $J = 0$. Thus $\fD(\fs) = I = \bigoplus_{V \in X} \fs^V$, which implies the isomorphism \eqref{eq:D(E7)-LA}.
\end{proof}

\subsection{Type \texorpdfstring{$E_8$}{E8}}

Let $G = W(E_8)$ be the reflection group of type $E_8$.

\begin{theorem} \label{thm:D(E8)-LA}
For the reflection group of type $E_8$, the map \eqref{eq:D(s)-factorization} defines an isomorphism
	\begin{equation} \label{eq:D(E8)-LA}
	\fD(\fs) \cong \fsl(V_{\refl}) \oplus \Big[ \bigoplus_{V \in E/\sim} \fsl(V) \Big] \oplus \Big[ \bigoplus_{V \in F} \osp(V) \Big].
	\end{equation}
\end{theorem}

\begin{proof}
Let $X = \set{V_{\refl}} \cup (E/\!\!\sim) \cup F$. Let $S_0 = \set{s_1,\ldots,s_7,s_8}$ be a set of Coxeter generators for $G$ such that the corresponding Coxeter graph is labeled as in \cite[Theorem 1.3.3]{Geck:2000}, and let $H \cong W(E_7)$ be the subgroup of $G$ generated by the set $\set{s_1,\ldots,s_7}$. The proof is entirely parallel to the arguments given previously for types $E_6$ and $E_7$, considering for each $V \in X$ the restricted module $\Res^G_H(V)$. The only novelty in this case is that the restricted module $\Res^G_H(V)$ may have one or more simple factors of multiplicity two.
\end{proof}

\begin{remark} \label{remark:E8-osp-dont-know}
For the simple modules of type $F$ in \eqref{eq:D(E8)-LA}, we do not know whether $\osp(V)$ is an orthogonal Lie algebra or a symplectic Lie algebra.
\end{remark}


\section{Lie superalgebras generated by reflections} \label{sec:preliminaries}

Throughout this section, let $G$ be a (not necessarily irreducible) finite reflection group with set of reflections $S$ as in Section \ref{subsec:generalities}. In Sections \ref{subsec:supergroup-prelim} and \ref{subsec:lie-superalgebra-prelim} we recall some preliminary results on supergroups and Lie super\-algebras; we adopt the general notational and terminological conventions for superspaces and superalgebras as laid out, for example, in \cite[\S2]{DK:2025}. In Section \ref{subsec:Lie-superalgebra-generalities}, we recall some general results on the structure of the Lie superalgebra generated by the set of reflections in $G$.

\subsection{Preliminaries on supergroups} \label{subsec:supergroup-prelim}

The sign character $\ve: G \to \set{\pm 1}$ makes $G$ into a supergroup with $\Gzero = \ker(\ve)$ and $\Gone = G \backslash \Gzero$. In particular, $S \subseteq \Gone$.\footnote{This creates the unfortunate notational confluence that $S_0 \subseteq \Gone$, but we trust the reader to remember in this context that, contrary to the subscript $0$, the set $S_0$ of simple reflections is concentrated in \emph{odd} superdegree.} This $\Z_2$-grading on $G$ extends by linearity to a $\Z_2$-grading on the group algebra $\CG$, making $\CG$ into a superalgebra. Recall the notation $\Irr(G) = \{ S^\lambda : \lambda \in \Lambda(G) \}$ for the set of isomorphism classes of the simple $\CG$-modules, with $(S^\lambda)^\ve = S^\lambda \otimes \ve \cong S^{\lambda'}$.

\begin{lemma}[{cf.\ \cite[Proposition 5.1]{Fulton:1991}}] \label{lemma:classical-Clifford}
Up to isomorphism, each simple $\CGzero$-module arises uniquely in one of the following two ways:
	\begin{enumerate}
	\item \label{item:type-E-simple} If $\lambda \neq \lambda'$, then $\Res_{\Gzero}^G(S^\lambda) \cong \Res_{\Gzero}^G(S^{\lambda'})$ is simple and isomorphic to its conjugate.
	\item \label{item:type-F-simple} If $\lambda = \lambda'$, then $\Res_{\Gzero}^G(S^{\lambda})$ is the direct sum of two simple, conjugate, non-iso\-mor\-phic $\CGzero$-sub\-modules $S^{\lambda^+}$ and $S^{\lambda^-}$, equal to the $+1$ and $-1$ eigen\-spaces of a choice of the map $\phi^{\lambda}$.
	\end{enumerate}
\end{lemma}

\begin{remark} \label{rem:S-lam-pm-to-mp}
If $\lambda = \lambda'$, then replacing $\phi^\lambda$ with $-\phi^\lambda$ swaps the labelings of $S^{\lambda^+}$ and $S^{\lambda^-}$.
\end{remark}

Define subsets $\calE(G)$ and $\calF(G)$ of $\Lambda(G)$ by
	\begin{equation} \label{eq:E-F-Irr-super} 
	\calE(G) = \{ \lambda \in \Lambda(G): \lambda \neq \lambda' \} \qquad \text{and} \qquad
	\calF(G) = \{ \lambda \in \Lambda(G): \lambda = \lambda' \}. 
	\end{equation}
Let $\sim$ be the equivalence relation on $\Lambda(G)$ generated by $\lambda \sim \lambda'$. Then $\Irr(\Gzero)$ is indexed by the set
	\begin{equation} \label{eq:G0-Irr-parameters}
	(\calE(G)/\!\!\sim) \cup \{ \lambda^+,\lambda^- : \lambda \in \calF(G) \}.
	\end{equation}

\begin{convention} \label{convention:even-iso-up-to-pm}
Given $V,W \in \Irr(\Gzero)$, write $V \simeq W$ if $V$ and $W$ are isomorphic as $\CGzero$-modules up to replacing one or more simple factors in $V$ of the form $S^{\lambda^\pm}$ for $\lambda \in \calF(G)$ with the corresponding ``opposite'' factor $S^{\lambda^{\mp}}$ (of the same dimension as $S^{\lambda^\pm}$) in $W$.
\end{convention}

\begin{remark} \label{rem:even-subgroup-branching}
Let $H$ be a reflection subgroup of $G$ (e.g., generated by a subset of the set of simple reflections in $G$). Let $\lambda \in \Lambda(G)$, and suppose for some non-negative integers $m_\mu$ that
	\[
	\Res_H^G(S^\lambda) = \bigoplus_{\mu \in \Lambda(H)} (S^\mu)^{\oplus m_\mu}.
	\]
	\begin{enumerate}

	\item If $\lambda \neq \lambda'$, then $\Res_{\Hzero}^{\Gzero}(S^\lambda) = \bigoplus_{\mu \in \Lambda(H)} (S^\mu)^{\oplus m_\mu}$ provides the $\CHzero$-module decomposition of the simple $\CGzero$-module $S^\lambda$, with the understanding that $S^\mu = S^{\mu^+} \oplus S^{\mu^-}$ if $\mu = \mu'$.
	
	\item Suppose $\lambda = \lambda'$. First observe that if $S^\mu$ occurs with multiplicity $m_\mu$ in $S^\lambda$, then $S^{\mu'} \cong (S^\mu)^\ve$ occurs with multiplicity $m_\mu$ in $(S^\lambda)^\ve \cong S^{\lambda'} = S^\lambda$. Thus, $m_\mu = m_{\mu'}$ for all $\mu \in \Lambda(H)$.
	
	Next let $t \in \Hone \subseteq \Gone$. On the one hand, the action of $t$ leaves invariant each $\CH$-module factor in $S^\lambda$. On the other hand, the action of $t$ maps the summands $S^{\lambda^+}$ and $S^{\lambda^-}$ of $S^\lambda$ one onto the other, and similarly for the summands $S^{\mu^+}$ and $S^{\mu^-}$ of $S^\mu$ for each $\mu \in \calF(H)$. This implies for each $\mu \in \calE(H)$ that the $\CHzero$-module multiplicity of $S^\mu \cong S^{\mu'}$ in $S^{\lambda^+}$ is the same as its multiplicity in $S^{\lambda^-}$, and implies for each $\mu \in \calF(H)$ that the multiplicity of $S^{\mu^\pm}$ in $S^{\lambda^+}$ is equal to the multiplicity of $S^{\mu^\mp}$ in $S^{\lambda^-}$.
	
	Combining the preceding observations, and applying Convention \ref{convention:even-iso-up-to-pm}, it follows that
		\begin{equation}
		\Res^{\Gzero}_{\Hzero}(S^{\lambda^\pm}) \simeq \Big[ \bigoplus_{\mu \in \calE(H)/\sim} (S^\mu)^{\oplus m_\mu} \Big] \oplus \Big[ \bigoplus_{\mu \in \calF(H)} (S^{\mu^\pm})^{\oplus m_\mu} \Big],
		\end{equation}
	where the sum over the set $\calE(H)/\!\!\sim$ means that we take the sum over a complete set of representatives in $\calE(H)$ for the distinct equivalence classes under the relation $\sim$.
	\end{enumerate}
\end{remark}

\begin{proposition}[{cf.\ \cite[Proposition 2.1.3]{DK:2026b}}] \label{prop:simple-supermodules}
Up to homogeneous isomorphism, each simple $\CG$-super\-module occurs uniquely in one of the following two ways:
	\begin{enumerate}
	\item \label{item:self-associate-G-supermodule} If $\lambda \neq \lambda'$, then there exists a Type Q simple $\CG$-supermodule $W^\lambda = W^{\lambda'}$, equipped with an odd involution $J^\lambda: W^\lambda \to W^\lambda$, such that $W^\lambda = S^\lambda \oplus S^{\lambda'}$ as an ungraded $\CG$-module, and $W_{\zero}^\lambda \cong S^\lambda \cong W_{\one}^\lambda$ as $\CGzero$-modules.

	\item If $\lambda = \lambda'$, then there exists a Type M simple $\CG$-supermodule $W^\lambda$ such that $W^\lambda = S^\lambda$ as an ungraded $\CG$-module, with $W_{\zero}^\lambda = S^{\lambda^+}$ and $W_{\one}^\lambda = S^{\lambda^-}$.
	\end{enumerate}
Each simple $\CG$-supermodule is uniquely determined up to an even isomorphism by its summands as a $\CGzero$-module and by the superdegrees in which those summands are concentrated.
\end{proposition}

Let $\Irr_s(G)$ be the set of isomorphism classes (up to homogeneous iso\-mor\-phism) of the simple $\CG$-super\-modules. Then $\Irr_s(G) = \Irr_Q(G) \cup \Irr_M(G)$, where $\Irr_Q(G)$ (resp., $\Irr_M(G)$) is the set of isomorphism classes of Type Q (resp., Type M) simple $\CG$-supermodules. The sets $\Irr_Q(G)$ and $\Irr_M(G)$ are indexed by the sets $(\calE(G)/\!\!\sim)$ and $\calF(G)$, respectively. In particular,
	\begin{equation} \label{eq:Irrs(G)-size}
	\abs{\Irr_s(G)} = \abs{ {\Lambda(G)/\!\!\sim} }
	\end{equation}
is the number of equivalence classes in $\Lambda(G)$ under the relation $\lambda \sim \lambda'$, or equivalently, the number of equivalence classes in $\Irr(G)$ under the relation $V \sim V^\ve$.

Given $W \in \Irr_s(G)$, let $\rho_W: \CG \to \End(W)$ denote the corresponding supermodule structure map. Recall that if $V$ is a vector superspace equipped with an odd involution $J$, then
	\begin{equation} \label{eq:Q(V)}
	Q(V) = Q(V,J) = \set{ \theta \in \End(V) : J \circ \theta = \theta \circ J}
	\end{equation}
is a simple subsuperalgebra of $\End(V)$.

\begin{corollary}[{cf.\ \cite[Corollary 2.1.5]{DK:2026b}}] \label{cor:super-artin-wedderburn-CG}
The sum of the structure maps $\rho_W$ induces a super\-algebra isomorphism
	\begin{equation} \label{eq:super-artin-wedderburn-CG}
	\CG \cong 
	\Big[ \bigoplus_{W \in \Irr_Q(G)} Q(W) \Big]
	\oplus 
	\Big[ \bigoplus_{W \in \Irr_M(G)} \End(W) \Big]. 
	\end{equation}
\end{corollary}

\subsection{Lie superalgebra preliminaries} \label{subsec:lie-superalgebra-prelim}

Given a superspace $V$, write $\gl(V)$ for the set $\End(V)$ considered as a Lie super\-algebra via the super commutator $[x,y] = xy - (-1)^{\ol{x} \cdot \ol{y}}yx$. Fixing a homogeneous basis for $V$, $\gl(V)$ identifies with the matrix Lie superalgebra
	\[
	\gl(m|n) := \set{ \left[	\begin{array}{c|c}
								A & B \\
								\hline
								C & D
								\end{array} \right] : A \in M_m(\C), B \in M_{m \times n}(\C), C \in M_{n \times m}(\C), D \in M_n(\C)}.
	\]
If $V$ is a superspace equipped with an odd involution $J : V \to V$, write $\fq(V)$ for the set $Q(V)$ considered as a Lie super\-algebra via the super commutator. If $\dim(\Vzero) = n$, then via an appropriate choice of basis one has
	\[
	\fq(V) \cong \fq(n) := \set{ \left[	\begin{array}{c|c}
							A & B \\
							\hline
							B & A
							\end{array} \right] : A \in M_n(\C), B \in M_n(\C) }.
	\]
For an arbitrary Lie super\-algebra $\g$, denote its derived subalgebra $[\g,\g]$ by $\fD(\g)$. Then $\fD(\gl(V))$ is the special linear Lie superalgebra $\fsl(V)$, and
	\[
	\fD(\glmn) = \fsl(m|n) := \set{ \left[	\begin{array}{c|c}
								A & B \\
								\hline
								C & D
								\end{array} \right] \in \glmn : \tr(A) - \tr(D) = 0}.
	\]
Let $\sq(V) = \fD(\fq(V))$. One can show that
	\[
	\fD(\fq(n)) = \sq(n) := \set{ \left[	\begin{array}{c|c}
								A & B \\
								\hline
								B & A
								\end{array} \right] \in \fq(n) : \tr(B) = 0 }.
	\]

\begin{lemma}[{\cite[Lemma 2.2.1]{DK:2026b}}] \label{lem:generated-by-odd} \ 
	\begin{enumerate}
	\item If $m \geq 2$, then $\fsl(m|m)$ is generated as a Lie superalgebra by $\fsl(m|m)_{\one}$.
	\item If $m \geq 3$, then $\sq(m)$ is generated as a Lie superalgebra by $\sq(m)_{\one}$.
	\end{enumerate}
\end{lemma}

Now write $\Lie(\CG)$ for the superalgebra $\CG$ considered as a Lie superalgebra via the super commutator $[x,y] = xy - (-1)^{\ol{x} \cdot \ol{y}}yx$, and set $\fD(\CG) = \fD(\Lie(\CG))$.\footnote{We trust that no serious confusion will arise between our use of this notation both here and in Section \ref{S:LA-by-reflections}.} Then \eqref{eq:super-artin-wedderburn-CG} restricts to Lie super\-algebra isomorphisms
	\begin{align} 
	\Lie(\CG) &\cong 
	\Big[ \bigoplus_{W \in \Irr_Q(G)} \fq(W) \Big]
	\oplus 
	\Big[ \bigoplus_{W \in \Irr_M(G)} \gl(W) \Big], \label{eq:LieCG} \\
	\fD(\CG) & \cong 
	\Big[ \bigoplus_{W \in \Irr_Q(G)} \sq(W) \Big] 
	\oplus 
	\Big[ \bigoplus_{W \in \Irr_M(G)} \fsl(W) \Big]. \label{eq:D(LieCG)}
	\end{align}
In particular, $\dim(\fD(\CG)) = \abs{G} - \abs{\Irr_Q(G)} - \abs{\Irr_M(G)} = \abs{G} - \abs{\Irr_s(G)}$.

	
\subsection{General structure of the Lie superalgebra generated by reflections} \label{subsec:Lie-superalgebra-generalities}

Now let $\g$ be the Lie subsuperalgebra of $\Lie(\CG)$ generated by the set $S$ of reflections in $G$.\footnote{In general, the subset $S_0$ of \emph{simple} reflections in $G$ generates a proper Lie sub-superalgebra of $\g$, in contrast to the situation of Remark \ref{remark:s-generated-by-S0}.} Our goal in the remainder of the paper is to show (for $G$ irreducible) that
	\begin{equation} \label{eq:LSA-goal}
	\g = \fD(\CG) + (\CSrm)^G.
	\end{equation}
In \cite[\S\S2.3--2.5]{DK:2026b}, we considered (for $G$ not necessarily irreducible) the six assumptions below. Note that \ref{item:S-union-conj-classes} is true by our assumptions on $G$, and $\gzero \subseteq \Lie(\CGzero)$ is an ordinary Lie algebra.
	\begin{enumerate}[label={(A\arabic*)}]
	\item \label{item:S-union-conj-classes} The set $S$ is a union of conjugacy classes in $G$, say $S = C_1 \cup \cdots \cup C_t$.

	\item \label{item:gzero-generates-CGzero} There exists $S' \subseteq \gzero$ such that $\Gzero = \subgrp{S'}$. In particular, $S'$ generates $\CGzero$ as an algebra.

	\item \label{item:S'-one-class-gen-G0} The set $S'$ consists of a single conjugacy class of order-two elements in $\Gzero$.

	\item \label{item:inductive-subgroup} There exist a subsupergroup $H \subseteq G$ and a subset $S_H \subseteq S$ such that:
		\begin{enumerate}
		\item \label{item:inductive-A1-A2} \ref{item:S-union-conj-classes} and \ref{item:gzero-generates-CGzero} hold for $H$ and the Lie subsuperalgebra $\fh \subseteq \Lie(\CH)$ generated by $S_H$;

		\item \label{item:inductive-D(hzero)} $\fh_{\zero} = \fD(\CH)_{\zero}$, and hence \eqref{eq:D(LieCG)} for the group $H$ restricts to an isomorphism
			\begin{equation} \label{eq:D(hzero)}
			\fD(\hzero) \cong \bigoplus_{V \in \Irr(\Hzero)} \fsl(V),
			\end{equation}
induced by the sum of the simple $\CHzero$-module structure maps $\rho_V: \CHzero \to \End(V)$;

		\item there are no simple $\CHzero$-modules of dimension $2$; and

		\item if $V$ is a simple $\CGzero$-module with $\dim(V) > 1$, then $\Res^{\Gzero}_{\Hzero}(V)$ contains some simple $\CHzero$-module $W$ with $\dim(W) > 1$.
		\end{enumerate}

	\item \label{item:D(gzero)-to-sl(V)} For each $V \in \Irr(\Gzero)$, the structure map $\rho_V: \CGzero \to \End(V)$ restricts to a surjective Lie algebra map $\rho_V: \fD(\gzero) \twoheadrightarrow \fsl(V)$.
	
	\item \label{item:no-G-simple-dim-2} There are no (ordinary, non-super) simple $\CG$-modules of dimension $2$.
	\end{enumerate}

Under combinations of these six assumptions, we then proved the following general results:
	
\begin{lemma}[{\cite[Lemma 2.3.1]{DK:2026b}}] \label{lemma:g-in-D(g)+span}
Suppose \ref{item:S-union-conj-classes} holds. Then
	\begin{equation} \label{eq:LSA-inclusion}
	\g \subseteq \fD(\CG) + (\CSrm)^G.
	\end{equation}
\end{lemma}

\begin{lemma}[{\cite[Lemma 2.3.2]{DK:2026b}}] \label{lemma:lie-algebra-center}
Suppose \ref{item:S-union-conj-classes} and \ref{item:gzero-generates-CGzero} hold. Then $Z(\gzero) = \gzero \cap Z(\CGzero)$, and the map $p: \CGzero \to Z(\CGzero)$, defined by $p(z) = \frac{1}{\abs{\Gzero}}  \sum_{g \in \Gzero} g z g^{-1}$, restricts to a map $p: \gzero \to Z(\gzero)$.
\end{lemma}

\begin{proposition}[{\cite[Proposition 2.3.5]{DK:2026b}}] \label{prop:g0-reductive}
Suppose \ref{item:S-union-conj-classes} and \ref{item:gzero-generates-CGzero} hold. Then the Lie algebra $\gzero$ is reductive, $\fD(\gzero)$ is semi\-simple, $\gzero = Z(\gzero) \oplus \fD(\gzero)$, and $\fD(\gzero) = \ker(p)$.
\end{proposition}

\begin{lemma}[{\cite[Lemmas 2.3.4 and 2.3.6]{DK:2026b}}] \label{lemma:simple-G0-g0-D(g0)}
Suppose \ref{item:S-union-conj-classes} and \ref{item:gzero-generates-CGzero} hold. Let $V$ be a $\CGzero$-module. The following are equivalent:
	\begin{enumerate}
	\item $V$ is simple as a $\CGzero$-module.
	\item $V$ is simple as a $\gzero$-module.
	\item $V$ is simple as a $\fD(\gzero)$-module.
	\end{enumerate}
\end{lemma}

\begin{proposition}[{\cite[Proposition 2.3.7]{DK:2026b}}] \label{prop:simples-equivalent-G0-g0}
Suppose \ref{item:S-union-conj-classes}, \ref{item:gzero-generates-CGzero}, and \ref{item:S'-one-class-gen-G0} hold. Let $V_1$ and $V_2$ be simple $\CGzero$-modules of dimensions greater than $1$. Then the following are equivalent:
	\begin{enumerate}
	\item \label{item:Gzero-iso} $V_1$ and $V_2$ are isomorphic as $\CGzero$-modules.
	\item \label{item:gzero-iso} $V_1$ and $V_2$ are isomorphic as modules over the Lie subalgebra $\gzero \subseteq \CGzero$.
	\item \label{item:D(gzero)-iso} $V_1$ and $V_2$ are isomorphic as modules over the Lie subalgebra $\fD(\gzero) \subseteq \CGzero$.
	\end{enumerate}
\end{proposition}

\begin{lemma}[{\cite[Lemma 2.5.2]{DK:2026b}}] \label{lemma:D(gzero)}
Suppose \ref{item:S-union-conj-classes}--\ref{item:D(gzero)-to-sl(V)} hold. Then $\fD(\gzero) = \fD(\fD(\CG)_{\zero})$, and the sum of the structure maps of the simple $\CGzero$-modules induces an isomorphism
	\begin{equation} \label{eq:D(D(CG)0)-AW-iso}
	\fD(\gzero) \cong \bigoplus_{V \in \Irr(\Gzero)} \fsl(V).
	\end{equation}
\end{lemma}

\begin{lemma}[{\cite[Lemma 2.5.3]{DK:2026b}}] \label{lemma:W(gn)}
Suppose \ref{item:S-union-conj-classes}--\ref{item:no-G-simple-dim-2} hold. Then for each $W \in \Irr_s(G)$, the image of $\g$ under the supermodule structure map $\rho_W: \CG \to \End(W)$ is
	\begin{equation} \label{eq:Wlambda(gn)}
	\rho_W(\g) = \begin{cases} 
	\sq(W) + \rho_W((\CSrm)^G) & \text{if $W$ is of Type Q,} \\
	\fsl(W) & \text{if $W$ is of Type M.}
	\end{cases}
	\end{equation}
\end{lemma}

\begin{remark} \label{rem:reverse-deductions}
In general, \ref{item:gzero-generates-CGzero} need not hold (e.g., see \cite[Example 4.1.4]{DK:2026b}), but the conclusions in Lemma \ref{lemma:lie-algebra-center} through Lemma \ref{lemma:W(gn)} can still be deduced retroactively if \eqref{eq:LSA-goal} is already known. By the main results of \cite{DK:2025,DK:2026b} and by \cite[\S2.4]{DK:2026b}, \eqref{eq:LSA-goal} is known to hold if $G$ is an irreducible finite reflection group of type $A_n$ ($n \geq 1$), $B_n$ ($n \geq 2$), $D_n$ ($n \geq 4$), or $I_2(m)$ ($m \geq 3$).

\end{remark}

The next result states that \ref{item:gzero-generates-CGzero} and \ref{item:S'-one-class-gen-G0} are satisfied if $G$ is irreducible of large enough rank.

\begin{lemma} \label{lemma:S'-generator}
Let $(G,S_0)$ be an irreducible Coxeter system of type $A_n$ ($n \geq 4$), $B_n$ ($n \geq 4$), $D_n$ ($n \geq 5$), $F_4$, $H_4$, or $E_n$ ($n=6,7,8$). Then there exists a conjugacy class $S'$ in $\Gzero$ that satisfies the assumptions \ref{item:gzero-generates-CGzero} and \ref{item:S'-one-class-gen-G0}.
\end{lemma}

\begin{proof}
Let $S_0 = \smallset{s_1,s_2,\ldots,s_n}$, and suppose the corresponding Coxeter graph is labeled as in \cite[Theorem 1.3.3]{Geck:2000}. Let $S'$ be the conjugacy class in $G$ of the element $s_1 s_n$. Then $(s_1s_n)^2 = 1$, so all elements of $S'$ are of order $2$. Further, $s_1 s_n = \frac{1}{2} \cdot (s_1 s_n + s_n s_1) = \frac{1}{2} \cdot [s_1,s_n] \in \gzero$. Since $S$ is a union of conjugacy classes in $G$, and since the operation of conjugation by an element $g \in G$ is an even linear map on $\Lie(\CG)$, it follows first that $\g$ and then that $\gzero$ is closed under conjugation by $G$. Then $S' \subseteq \gzero$. Commutativity of $s_1$ and $s_n$ also implies that $s_1 (s_1s_n) = (s_1s_n) s_1$, so $C_G(s_1s_n) \not\subseteq \Gzero$. Hence anything conjugate to $s_1s_n$ by an element of $G$ is also conjugate by an element of $\Gzero$ and  $S'$ remains a single conjugacy class in $\Gzero$.

It remains to show that $\Gzero = \subgrp{S'}$. Since $S'$ is a conjugacy class in $G$, we have that $\subgrp{S'}$ is normal in $G$. It is an exercise to show that $\Gzero = \ker(\ve) = \subgrp{S''}$, where $S'' = \smallset{ s_1s_i : 2 \leq i \leq n}$, so we will show that $S'' \subseteq \subgrp{S'}$. Write $g \sim g'$ if $g$ and $g'$ are conjugate in $G$. For arbitrary $s_i,s_j \in S_0$, observe that if $m_{i,j}$ is odd, with say $m_{i,j} = 2k+1$, then the identity $(s_is_j)^{m_{i,j}} = 1$ implies that
	\[
	s_j = (s_is_j)^k s_i (s_js_i)^k = (s_is_j)^k \cdot s_i \cdot (s_is_j)^{-k},
	\]
i.e., $s_i \sim s_j$. Further, if $s_\ell$ commutes with each of $s_i$ and $s_j$, then $s_j s_\ell = (s_is_j)^k (s_is_\ell) (s_is_j)^{-k}$, so $s_i s_\ell \sim s_j s_\ell$. The details of the argument now vary slightly depending on the Coxeter type of $G$:
	\begin{enumerate}
	\item Type $A_n$ ($n \geq 4$): One has $s_1 s_n \sim s_1s_{n-1} \sim s_1 s_{n-2} \sim \cdots \sim s_1 s_3$, and also $s_1 s_n \sim s_2 s_n$. Since $s_1s_n \in S'$ and $s_2 s_n \in S'$, we have $s_1 s_2 = (s_1s_n)(s_2s_n)^{-1} \in \subgrp{S'}$. Thus $S'' \subseteq \subgrp{S'}$.

	\item Type $B_n$ ($n \geq 4$): One has $s_1 s_n \sim s_1s_{n-1} \sim \cdots \sim s_1 s_3$. Since $s_1 s_3 \in S'$ and $s_1s_4 \in S'$, we have $s_3s_4 = (s_1s_3)^{-1} (s_1s_4) \in \subgrp{S'}$. Now $s_3s_4 \sim s_2s_4$, so $s_2 s_4 \in \subgrp{S'}$. Since $s_1 s_4 \in S'$ and $s_2 s_4 \in \subgrp{S'}$, we have $s_1s_2 = (s_1s_4)(s_2s_4)^{-1} \in \subgrp{S'}$. Thus $S'' \subseteq \subgrp{S'}$.

	\item Type $D_n$ ($n \geq 5$). One has $s_1s_n \sim s_1s_{n-1} \sim \cdots \sim s_1 s_4$, and also $s_1s_n \sim s_3 s_n \sim s_2 s_n$. Since $s_1 s_n \in S'$ and $s_2 s_n \in S'$, we have $s_1 s_2 = (s_1s_n)(s_2s_n)^{-1} \in \subgrp{S'}$. Similarly, since $s_1 s_n \in S'$ and $s_3 s_n \in S'$, we have $s_1 s_3 \in \subgrp{S'}$. Thus $S'' \subseteq \subgrp{S'}$.
	
	\item Types $F_4$ and $H_4$: One has $s_1 s_4 \sim s_1 s_3$ and $s_1 s_4 \sim s_2 s_4$. Since $s_1 s_4 \in S'$ and $s_2 s_4 \in S'$, we have $s_1 s_2 = (s_1s_4) (s_2s_4)^{-1} \in \subgrp{S'}$. Thus $S'' \subseteq \subgrp{S'}$.
	
	\item Type $E_n$ ($n=6,7,8$): One has $s_1 s_n \sim s_1 s_{n-1} \sim \cdots \sim s_1 s_4$, and $s_1 s_4 \sim s_1 s_2 \sim s_3 s_2 = s_2 s_3$. Since $s_1 s_2 \in S'$ and $s_3 s_2 \in S'$, we have $s_1 s_3 = (s_1 s_2)(s_2 s_3) \in \subgrp{S'}$. Then $S'' \subseteq \subgrp{S'}$. \qedhere
	\end{enumerate}
\end{proof}

\section{Lie superalgebra calculations for groups of exceptional type} \label{S:Lie-superalgebra-calcs}

\subsection{Type \texorpdfstring{$I_2(m)$ for $m \geq 3$}{I2(m) for m >=3}}

Recall the notation of Section \ref{subsec:I2m-LA}. The following result is stated and proved in \cite[\S2.4]{DK:2026b}, but we include it here as well for the sake of completeness. The proof is by dimension comparison, as for Theorem \ref{thm:D(s(Dm)-LA}.

\begin{theorem} \label{thm:D(Dih)-LSA}
For the dihedral group $G = D_m$ of order $2m$ ($m \geq 3$), the Lie superalgebra inclusion  \eqref{eq:LSA-inclusion} is an equality, i.e., $\g = \fD(\CD_m) + (\CSrm)^G$.
\end{theorem}

\subsection{Type \texorpdfstring{$H_3$}{H3}}

\begin{theorem} \label{thm:D(H3)-LSA}
For the reflection group $G = W(H_3)$ of type $H_3$, the Lie superalgebra inclusion \eqref{eq:LSA-inclusion} is an equality, i.e., $\g = \fD(\CG) + (\CSrm)^G = \fD(\CG) + \Span\smallset{\calC_\ell}$.
\end{theorem}

\begin{proof}
Let $G = W(H_3)$ be the Coxeter group of type $H_3$, which has order $120$. A list of the simple $\CG$-modules and their dimensions is given in Table \ref{tab:H3}. From the table one deduces that there are $5$ simple $\CG$-supermodules of Type Q, and none of Type M. Then \eqref{eq:LSA-inclusion} and \eqref{eq:D(LieCG)} imply that $\dim(\g) \leq 116$. On the other hand, using GAP \cite{GAP4} we have calculated that $\dim(\fg) = 116$. Thus \eqref{eq:LSA-inclusion} must be an equality.
\end{proof}

\subsection{Type \texorpdfstring{$H_4$}{H4}}

\begin{theorem} \label{thm:D(H4)-LSA}
For the reflection group $G = W(H_4)$ of type $H_4$, the Lie superalgebra inclusion \eqref{eq:LSA-inclusion} is an equality, i.e., $\g = \fD(\CG) + (\CSrm)^G = \fD(\CG) + \Span\smallset{\calC_\ell}$.
\end{theorem}

\begin{proof}
Let $G = W(H_4)$ be the Coxeter group of type $H_4$, which has order $14400$. A list of the simple $\CG$-modules and their dimensions is given in Table \ref{tab:H4}. From the table one deduces that there are $9$ simple $\CG$-supermodules of Type Q, and $16$ of Type M. Then \eqref{eq:LSA-inclusion} and \eqref{eq:D(LieCG)} imply that $\dim(\g) \leq 14376$. On the other hand, using GAP \cite{GAP4} we have calculated\footnote{Our calculation took around 54 hours to complete.} that $\dim(\fg) = 14376$. Thus \eqref{eq:LSA-inclusion} must be an equality.
\end{proof}

\subsection{Type \texorpdfstring{$F_4$}{F4}}

\begin{theorem} \label{thm:D(F4)-LSA}
For the reflection group $G = W(F_4)$ of type $F_4$, the Lie superalgebra inclusion \eqref{eq:LSA-inclusion} is an equality, i.e., $\g = \fD(\CG) + (\CSrm)^G = \fD(\CG) + \Span\smallset{\calC_\ell,\calC_s}$.
\end{theorem}

\begin{proof}
Let $G = W(F_4)$ be the Coxeter group of type $F_4$, which has order $1152$. A list of the simple $\CG$-modules and their dimensions is given in Table \ref{tab:F4}. From the table one deduces that there are $10$ simple $\CG$-super\-modules of Type Q, and $5$ of Type M. Then \eqref{eq:LSA-inclusion} and \eqref{eq:D(LieCG)} imply that $\dim(\g) \leq 1139$. On the other hand, using GAP \cite{GAP4} we have calculated that $\dim(\fg) = 1139$. Thus \eqref{eq:LSA-inclusion} must be an equality.
\end{proof}

\subsection{Type \texorpdfstring{$E_6$}{E6}}

Let $G = W(E_6)$ be the reflection group of type $E_6$. Let $S_0 = \set{s_1,\ldots,s_5,s_6}$ be a set of Coxeter generators for $G$ such that the corresponding Coxeter graph is labeled as in \cite[Theorem 1.3.3]{Geck:2000}, and let $H \cong W(D_5)$ be the subgroup of $G$ generated by the set $\set{s_1,\ldots,s_5}$. Let $S_H$ be the set of reflections in $H$.

\begin{lemma} \label{lemma:A1-A6-E6-D5}
With notation as in the previous paragraph, conditions \ref{item:S-union-conj-classes}--\ref{item:no-G-simple-dim-2} are satisfied for the group $G$ relative to the subgroup $H$.
\end{lemma}

\begin{proof}
Condition \ref{item:S-union-conj-classes} is satisfied because $G$ is a reflection group, \ref{item:gzero-generates-CGzero} and \ref{item:S'-one-class-gen-G0} are satisfied by Lemma \ref{lemma:S'-generator}, and \ref{item:no-G-simple-dim-2} holds by inspection. Parts (a) and (b) of \ref{item:inductive-subgroup} are true because $H$ is a reflection group, and by  \cite[Lemmas 5.5.2 and 5.5.3]{DK:2026b} together with the explanation in the first paragraph of the proof of \cite[Lemma 2.5.2]{DK:2026b}. Part (c) of \ref{item:inductive-subgroup} follows from Lemma \ref{lemma:classical-Clifford}, since by observation (see Table \ref{tab:D5}) there are no simple $\CH$-modules of dimension $2$, and there is no simple $\CH$-module $V$ of dimension $4$ such that $V \cong V^\ve$. Part (d) of \ref{item:inductive-subgroup} can be verified in a similar fashion by observing for each $V \in \Irr(G)$ the composition factor multiplicities of $\Res^G_H(V)$ (which can be explicitly determined by character calculations in GAP) and by applying Remark \ref{rem:even-subgroup-branching}.

It remains to verify \ref{item:D(gzero)-to-sl(V)}. Let $V \in \Irr(\Gzero)$; we may evidently assume that $\dim(V) > 1$. Either $V = S^\lambda$ for some $\lambda \in \calE(G)$, or $V = S^{\lambda^\pm}$ for some $\lambda \in \calF(G)$. Set $\g_V = \rho_V(\fD(\gzero))$ and $\fh_V = \rho_V(\fD(\hzero))$. Then $\fh_V \subseteq \g_V$ are semi\-simple Lie algebras by Proposition \ref{prop:g0-reductive} (for $\g_V$) and part (b) of \ref{item:inductive-subgroup} (for $\fh_V$), and $\g_V$ acts irreducibly on $V$ by Lemma \ref{lemma:simple-G0-g0-D(g0)}. It follows from \eqref{eq:D(hzero)} that,
	\begin{equation} \label{eq:hV}
	\fh_V \cong \bigoplus_{\substack{W \in \Irr(\Hzero) \\ W \leq V}} \fsl(W),
	\end{equation}
i.e., $\fh_V$ consists of a copy of $\fsl(W)$ for each $W \in \Irr(\Hzero)$ that occurs with multiplicity at least one in $V$. By Remark \ref{rem:even-subgroup-branching}, the dimensions of the $\CHzero$-module composition factors in $V$ can be deduced from the dimensions of the $\CH$-module composition factors in $\Res^G_H(S^\lambda)$, which in turn can be explicitly computed by character calculations in GAP. Then using Lemma~\ref{lemma:simple-ideal-criterion} one can show that $\g_V$ is a simple Lie algebra. Next, applying \eqref{eq:hV} and Remark \ref{rem:even-subgroup-branching} again, one can compute the rank of the semisimple Lie algebra $\fh_V$. By observation, one sees that $2 \cdot \rk(\fh_V) > \dim(V)$. Since $\fh_V \subseteq \g_V$, then $\rk(\fh_V) \leq \rk(\g_V)$, so this implies that $\g_V = \fsl(V)$, by \cite[Lemme 13]{Marin:2007}.
\end{proof}

\begin{theorem} \label{thm:D(E6)-LSA}
For the reflection group $G = W(E_6)$ of type $E_6$, the Lie superalgebra inclusion \eqref{eq:LSA-inclusion} is an equality, i.e., $\g = \fD(\CG) + (\CSrm)^G = \fD(\CG) + \Span\smallset{\calC_\ell}$.
\end{theorem}

\begin{proof}
By \eqref{eq:LSA-inclusion} one has $\g \subseteq \fD(\CG) + \C \cdot \calC_\ell$, and $\calC_\ell \in \g$ because it is the sum of the generators for $\g$, so it remains to show that $\fD(\CG) \subseteq \g$. In the rest of the proof we identify $\CG$ with its image under the isomorphism \eqref{eq:LieCG}.

First we will show that $\fD(\CG)_{\one} \subseteq \fg$. By \eqref{eq:D(LieCG)}, one has
	\begin{align*}
	\fD(\CG)_{\one} &= 
	\Big[ \bigoplus_{W \in \Irr_Q(G)} \sq(W)_{\one} \Big] 
	\oplus 
	\Big[ \bigoplus_{W \in \Irr_M(G)} \fsl(W)_{\one} \Big] \\
	&\cong 
	\Big[ \bigoplus_{\lambda \in \calE(G)/\sim} \fsl(S^\lambda) \Big]
	\oplus 
	\Big[ \bigoplus_{\lambda \in \calF(G)} \Hom(S^{\lambda^+},S^{\lambda^-}) \oplus \Hom(S^{\lambda^-},S^{\lambda^+}) \Big],
	\end{align*}
and by Lemmas \ref{lemma:A1-A6-E6-D5} and \ref{lemma:D(gzero)}, one has
	\begin{align*}
	\fD(\gzero) &\cong 
	\Big[ \bigoplus_{\lambda \in \calE(G)/\sim} \fsl(S^\lambda) \Big] 
	\oplus 
	\Big[ \bigoplus_{\lambda \in \calF(G)} \fsl(S^{\lambda^+}) \oplus \fsl(S^{\lambda^-}) \Big].
	\end{align*}
Then $\fD(\CG)_{\one}$ is a direct sum of pairwise non-isomorphic simple modules for the semisimple Lie algebra $\fD(\gzero)$. Since $\gone$ is a $\fD(\gzero)$-submodule of $\fD(\CG)_{\one} + \C \cdot \calC_\ell$, it must contain some subset of the summands in $\fD(\CG)_{\one}$. Using Lemma \ref{lemma:W(gn)}, we see that each of these summands is contained in the image of the corresponding projection map $\rho_W: \g \to \End(W)$, and, hence, must be contained in $\gone$. Thus $\fD(\CG)_{\one} \subseteq \g$.

Denote by $\bbW$ the Type Q simple $\CG$-supermodule that is equal as an ungraded $\CG$-module to $\C_{\triv} \oplus \ve$, the direct sum of the trivial $\CG$-module and the sign character $\ve$. From the dimensions of the simple $\CG$-modules, one observes that $\dim(W) \geq 10$ for all $W \in \Irr_M(G)$, and $\dim(W) \geq 12$ for all $W \in \Irr_Q(G)$ except for $W = \bbW$. Then applying Lemma \ref{lem:generated-by-odd}, we deduce that $\fsl(W) \subseteq \g$ for all $W \in \Irr_M(G)$, and $\sq(W) \subseteq \g$ for all $W \in \Irr_Q(G)$ except perhaps for $W = \bbW$.

It remains to show that $\sq(\bbW) \subseteq \g$. In this case one has $\sq(\bbW) = \C \cdot \id_{\bbW}$. If $\tau$ is any reflection in $G$, then $\id_{\CG} = \tau^2 = \frac{1}{2} \cdot [\tau,\tau] \in \g$. Under the isomorphism of Corollary \ref{cor:super-artin-wedderburn-CG}, one has $1_{\CG} = \sum_{W \in \Irr_s(G)} \id_W$. Let
	\[
	\Psi = 1_{\CG} - \Big( \sum_{\substack{W \in \Irr_s(G) \\ \dim(W) > 2}} \id_W \Big).
	\]
Then $\Psi \in \g$, because it is a linear combination of elements known to be in $\g$, and $\Psi$ acts as zero on all simple $\CG$-supermodules except for $\bbW$, on which it acts as the identity. In other words, $\Psi = \id_{\bbW}$. Then $\sq(\bbW) \subseteq \g$, and we conclude that $\fD(\CG) \subseteq \g$.
\end{proof}

\subsection{Type \texorpdfstring{$E_7$}{E7}}

Let $G = W(E_7)$ be the reflection group of type $E_7$. Let $S_0 = \set{s_1,\ldots,s_6,s_7}$ be a set of Coxeter generators for $G$ such that the corresponding Coxeter graph is labeled as in \cite[Theorem 1.3.3]{Geck:2000}, and let $H \cong W(E_6)$ be the subgroup of $G$ generated by the set $\set{s_1,\ldots,s_6}$. Let $S_H$ be the set of reflections in $H$.

\begin{lemma} \label{lemma:A1-A6-E7-E6}
With notation as in the previous paragraph, conditions \ref{item:S-union-conj-classes}--\ref{item:no-G-simple-dim-2} are satisfied for the group $G$ relative to the subgroup $H$.
\end{lemma}

\begin{proof}
The proof is entirely parallel to the argument given previously for type $E_6$.
\end{proof}

\begin{theorem} \label{thm:D(E7)-LSA}
For the reflection group $G = W(E_7)$ of type $E_7$, the Lie superalgebra inclusion \eqref{eq:LSA-inclusion} is an equality, i.e., $\g = \fD(\CG) + (\CSrm)^G = \fD(\CG) + \Span\smallset{\calC_\ell}$.
\end{theorem}

\begin{proof}
The proof is entirely parallel to the argument given previously for type $E_6$.
\end{proof}

\subsection{Type \texorpdfstring{$E_8$}{E8}}

Let $G = W(E_8)$ be the reflection group of type $E_8$. Let $S_0 = \set{s_1,\ldots,s_7,s_8}$ be a set of Coxeter generators for $G$ such that the corresponding Coxeter graph is labeled as in \cite[Theorem 1.3.3]{Geck:2000}, and let $H \cong W(E_7)$ be the subgroup of $G$ generated by the set $\set{s_1,\ldots,s_7}$. Let $S_H$ be the set of reflections in $H$.

\begin{lemma} \label{lemma:A1-A6-E8-E7}
With notation as in the previous paragraph, conditions \ref{item:S-union-conj-classes}--\ref{item:no-G-simple-dim-2} are satisfied for the group $G$ relative to the subgroup $H$.
\end{lemma}

\begin{proof}
The proof is entirely parallel to the arguments given previously for types $E_6$ and $E_7$.
\end{proof}

\begin{theorem} \label{thm:D(E8)-LSA}
For the reflection group $G = W(E_8)$ of type $E_8$, the Lie superalgebra inclusion \eqref{eq:LSA-inclusion} is an equality, i.e., $\g = \fD(\CG) + (\CSrm)^G = \fD(\CG) + \Span\smallset{\calC_\ell}$.
\end{theorem}

\begin{proof}
The proof is entirely parallel to the arguments given previously for types $E_6$ and $E_7$.
\end{proof}

\appendix

\section{Information on irreducible representations} \label{sec:info-irreps}

Let $G$ be a finite reflection group as in Section \ref{subsec:generalities}, with sign representation $\ve$. Set $V^\ve = V \otimes \ve$, and write $\C_{\triv}$ for the one-dimensional trivial $G$-module.

\subsection{Type \texorpdfstring{$H_3$}{H3}}

Let $G = W(H_3)$ be the Coxeter group of type $H_3$. A list of the simple $\CG$-modules and their dimensions is given in Table \ref{tab:H3}; cf.\ \cite[Table C.1]{Geck:2000}.

\begin{table}[htbp]
\caption{Type $H_3$} \label{tab:H3}
\begin{tabular}{llc}
\hline
$U$			& $U^\ve$	& $\dim$ \\
\hline
$U_1 = \C_{\triv}$		& $U_2 = \ve$		& 1 \\
\hline
$U_3 = \Lambda^2(U_5)$		& $U_5 = U_{\refl}$		& 3 \\
\hline
$U_4$		& $U_6$		& 3 \\
\hline
$U_7$		& $U_8$		& 4 \\
\hline
$U_9$		& $U_{10}$	& 5 \\
\hline
\end{tabular}
\end{table}

\subsection{Type \texorpdfstring{$H_4$}{H4}}

Let $G = W(H_4)$ be the Coxeter group of type $H_4$. A list of the simple $\CG$-modules and their dimensions is given in Table \ref{tab:H4}; cf.\ \cite[Table C.2]{Geck:2000}. Recall for type $H_4$ that each simple $\CG$-module can be realized over $\Q(\sqrt{5})$; see \cite[5.3.8]{Geck:2000}. Let $\sigma : \C \to \C$ be a field automorphism lifting the nontrivial automorphism of $\Q(\sqrt{5})$. Given a $\CG$-module $V$, let ${}^\sigma V = \C \otimes_{\sigma} V$ be the $\CG$-module obtained from $V$ by `scalar extension' along $\sigma$. Writing ${}^\sigma v = 1 \otimes_\sigma v$ for $v \in V$, one has $g.({}^\sigma v) = {}^\sigma(g.v)$ for $g \in G$, and $\lambda \cdot {}^\sigma v = {}^\sigma( \sigma^{-1}(\lambda) \cdot v)$ for $\lambda \in \C$.

\begin{table}[htbp]
\caption{Type $H_4$} \label{tab:H4}
\begin{tabular}[t]{llc}
\hline
$V$			& $V^\ve$	& $\dim$ \\
\hline
$V_1 = \C_{\triv}$		& $V_2 =\ve$		& 1 \\
\hline
$V_3 = V_{\refl}$		& $V_5 = \Lambda^3(V_3)$		& 4 \\
\hline
$V_4 = {}^\sigma V_{\refl}$		& $V_6 = \Lambda^3(V_4)$		& 4 \\
\hline
$V_7 = \Lambda^2(V_4)$		& $V_7$		& 6 \\
\hline
$V_8 = \Lambda^2(V_3)$		& $V_8$		& 6 \\
\hline
$V_9$		& $V_9$		& 8 \\
\hline
$V_{10}$	& $V_{10}$	& 8 \\
\hline
$V_{11}$	& $V_{13}$	& 9 \\
\hline
$V_{12}$	& $V_{14}$	& 9 \\
\hline
\end{tabular} \qquad
\begin{tabular}[t]{llc}
\hline
$V$			& $V^\ve$	& $\dim$ \\
\hline
$V_{15}$	& $V_{15}$	& 10 \\
\hline
$V_{16}$	& $V_{17}$	& 16 \\
\hline
$V_{18}$	& $V_{19}$	& 16 \\
\hline
$V_{20}$	& $V_{20}$	& 16 \\
\hline
$V_{21}$	& $V_{21}$	& 16 \\
\hline
$V_{22}$	& $V_{22}$	& 18 \\
\hline
$V_{23}$	& $V_{23}$	& 24 \\
\hline
$V_{24}$	& $V_{24}$	& 24 \\
\hline
$V_{25}$	& $V_{25}$	& 24 \\
\hline
\end{tabular} \qquad
\begin{tabular}[t]{llc}
\hline
$V$			& $V^\ve$	& $\dim$ \\
\hline
$V_{26}$	& $V_{26}$	& 24 \\
\hline
$V_{27}$	& $V_{28}$	& 25 \\
\hline
$V_{29}$	& $V_{29}$	& 30 \\
\hline
$V_{30}$	& $V_{30}$	& 30 \\
\hline
$V_{31}$	& $V_{32}$	& 36 \\
\hline
$V_{33}$	& $V_{33}$	& 40 \\
\hline
$V_{34}$	& $V_{34}$	& 48 \\
\hline
\end{tabular}
\end{table}

\subsection{Type \texorpdfstring{$F_4$}{F4}}

Let $G = W(F_4)$ be the Coxeter group of type $F_4$. A list of the simple $\CG$-modules and their dimensions is given in Table \ref{tab:F4}; cf.\ \cite[Table C.3]{Geck:2000}.

\begin{table}[htbp]
\caption{Type $F_4$} \label{tab:F4}
\begin{tabular}[t]{llc}
\hline
$V$			& $V^\ve$	& $\dim$ \\
\hline
$V_1 = \C_{\triv}$		& $V_2 = \ve$		& 1 \\
\hline
$V_3$		& $V_4$		& 1 \\
\hline
$V_5$		& $V_6$		& 2 \\
\hline
$V_7$		& $V_8$		& 2 \\
\hline
$V_9 = V_{\refl}$		& $V_{12} = \Lambda^3(V_9)$	& 4 \\
\hline
$V_{10}$	& $V_{11}$	& 4 \\
\hline
$V_{13}$	& $V_{13}$	& 4 \\
\hline
$V_{14} = \Lambda^2(V_9)$	& $V_{14}$	& 6 \\
\hline
\end{tabular} \qquad
\begin{tabular}[t]{llc}
\hline
$V$			& $V^\ve$	& $\dim$ \\
\hline
$V_{15}$	& $V_{15}$	& 6 \\
\hline
$V_{16}$	& $V_{17}$	& 8 \\
\hline
$V_{18}$	& $V_{19}$	& 8 \\
\hline
$V_{20}$	& $V_{23}$	& 9 \\
\hline
$V_{21}$	& $V_{22}$	& 9 \\
\hline
$V_{24}$	& $V_{24}$	& 12 \\
\hline
$V_{25}$	& $V_{25}$	& 16 \\
\hline
\end{tabular}
\end{table}

\subsection{Type \texorpdfstring{$D_5$}{D5}}

Let $G = W(D_5)$ be the Coxeter group of type $D_5$. A list of the simple $\CG$-modules and their dimensions is given in Table \ref{tab:D5}; the notation follows that in \cite[\S3.2]{DK:2026a}.

\begin{table}[htbp]
\caption{Type $D_5$} \label{tab:D5}
\begin{tabular}{llc}
\hline
$U$			& $U^\ve$	& $\dim$ \\
\hline
$U_1 = S^{\{[5],\emptyset \}} = \C_{\triv}$	& $U_2 = S^{\{\emptyset,[1^5] \}}$	& 1 \\
\hline
$U_3 = S^{\{[41],\emptyset \}}$				& $U_4 = S^{\{\emptyset,[21^3] \}} = \Lambda^3(U_3)$	& 4 \\
\hline
$U_5 = S^{\{[32],\emptyset \}}$				& $U_6 = S^{\{\emptyset,[2^21] \}}$	& 5 \\
\hline
$U_7 = S^{\{[4],[1] \}} = U_{\refl}$		& $U_8 = S^{\{[1],[1^4] \}} = \Lambda^4(U_7)$	& 5 \\
\hline
$U_9 = S^{\{[31^2],\emptyset \}} = \Lambda^2(U_3)$	& $U_9$			& 6 \\
\hline
$U_{10} = S^{\{[3],[2] \}}$					& $U_{12} = S^{\{[1^2],[1^3] \}}$			& 10 \\
\hline
$U_{11} = S^{\{[3],[1^2] \}} = \Lambda^3(U_7)$			& $U_{13} = S^{\{[2],[1^3] \}} = \Lambda^2(U_7)$			& 10 \\
\hline
$U_{14} = S^{\{[2^2],[1] \}}$				& $U_{14}$			& 10 \\
\hline
$U_{15} = S^{\{[31],[1] \}}$				& $U_{16} = S^{\{[1],[21^2] \}}$			& 15 \\
\hline
$U_{17} = S^{\{[21],[2] \}}$					& $U_{18} = S^{\{[1^2],[21] \}}$			& 20 \\
\hline
\end{tabular}
\end{table}

\makeatletter
\renewcommand*{\@biblabel}[1]{\hfill#1.}
\makeatother

\bibliographystyle{eprintamsplain}
\bibliography{LA-LSA-reflections-exceptional-types}

\end{document}